\documentclass[a4paper,12pt]{article}

\usepackage[margin=1.1in]{geometry}
\usepackage{amsmath,amssymb,amsthm,mathtools}
\usepackage{algorithm}
\usepackage{comment}
\usepackage{xcolor}
\usepackage{hyperref}

\usepackage{booktabs,tabularx}

\usepackage{algpseudocode}
\newtheorem{definition}{Definition}
\newtheorem{theorem}{Theorem}

\newtheorem{remark}{Remark}
\newtheorem{lemma}{Lemma}

\def\x{{\mathbf{x}}}

\def\c{{\mathbf{c}}}
\def\g{{\mathbf{g}}}
\def\e{{\mathbf{e}}}

\def\a{{\mathbf{a}}}
\def\u{{\mathbf{u}}}
\def\v{{\mathbf{v}}}
\def\z{{\mathbf{z}}}
\def\w{{\mathbf{w}}}

\def\h{{\mathbf{h}}}
\def\y{{\mathbf{y}}}
\def\q{{\mathbf{q}}}
\def\p{{\mathbf{p}}}
\def\b{{\mathbf{b}}}

\def\X{{\mathbf{X}}}
\def\Y{{\mathbf{Y}}}
\def\A{{\mathbf{A}}}
\def\M{{\mathbf{M}}}
\def\I{{\mathbf{I}}}

\def\V{{\mathbf{V}}}

\def\U{{\mathbf{U}}}
\def\Q{{\mathbf{Q}}}
\def\P{{\mathbf{P}}}

\newcommand{\mA}{\mathcal{A}}
\newcommand{\R}{\mathcal{R}}
\newcommand{\mX}{\mathcal{X}}
\newcommand{\mV}{\mathcal{V}}

\newcommand{\mC}{\mathcal{C}}

\newcommand{\mK}{\mathcal{K}}
\newcommand{\mS}{\mathcal{S}}

\newcommand{\mB}{\mathcal{B}}
\newcommand{\mbS}{\mathbb{S}}
\newcommand{\mD}{\mathcal{D}}
\newcommand{\E}{\mathbb{E}}
\newcommand{\Id}{\textbf{I}}

\newcommand{\dist}{\textrm{dist}}
\newcommand{\nnz}{\textrm{nnz}}
\newcommand{\conv}{\textrm{conv}}

\newcommand{\rank}{\textrm{rank}}
\newcommand{\reals}{\mathbb{R}}

\newcommand{\supp}{\mathop{\mbox{\rm supp}}}
\newcommand{\dom}{\mathop{\mbox{\rm dom}}}

\DeclareMathOperator*{\argmin}{argmin}

\def\c{{\mathbf{c}}}
\def\s{{\mathbf{s}}}
\newcommand{\mY}{\mathcal{Y}}

\newcommand{\simplex}{\Delta}

\newcommand{\rr}{{\mathbf{r}}}

\newcommand{\ip}[2]{\left\langle #1,#2\right\rangle}
\newcommand{\norm}[1]{\left\lVert #1\right\rVert}

\title{Complexities of Weak Proximal Oracle Methods \\ for Composite Convex Optimization}
\author{Dan Garber\\  \small{dangar@technion.ac.il}}
\date{\vspace{-5pt} Faculty of Data and Decision Sciences \vspace{3pt} \\ Technion - Israel Institute of Technology}

\begin{document}
\maketitle

\begin{abstract}
We consider a standard convex composite optimization problem with either smooth or nonsmooth objective function, and under quadratic growth. In recent years, several works gave algorithms based on a \textit{weak proximal oracle} (WPO) that essentially match in oracle complexities  proximal (sub)gradient methods relying on exact prox operations. Importantly, such WPOs, which relax the strong optimality condition of the standard prox operator, may admit much more efficient implementation in terms of runtime when  optimal solutions have some sparse structure.  A question remained if such WPO-based methods can be accelerated (in the sense of Nesterov's accelerated gradient). In this work we provide a negative answer by establishing lower bounds against both deterministic and randomized methods. Thus, while WPOs can substantially reduce the cost of individual oracle calls, this comes with an inherent loss in oracle complexity. We also provide a new upper-bound for WPO-based nonsmooth convex composite optimization, nearly matching the proximal subgradient method. 

\end{abstract}

\section{Introduction}
We consider the standard composite model
\begin{align}\label{eq:optProb}
\min_{\x\in\mathbb{V}}\left\{F(\x):=f(\x)+\R(\x)\right\}.
\end{align}
Throughout, $\mathbb{V}$ is a finite-dimensional Euclidean space of ambient
dimension $n:=\dim(\mathbb{V})$.  We write $\ip{\cdot}{\cdot}$ and
$\norm{\cdot}$ for its inner product and induced Euclidean norm, respectively. For $\x\in\mathbb V$ and a
nonempty set $\mS\subseteq\mathbb V$ define $\dist(\x,\mS):=\inf_{\z\in\mS}\norm{\x-\z}$.

The function $f:\mathbb{V}\to\reals$ is closed and convex, and
$\R:\mathbb{V}\to\reals\cup\{+\infty\}$ is proper, closed, and convex.  Denote
$\mX^*:=\argmin_{\x\in\mathbb V}F(\x)$, $F^*:=\min_{\x\in\mathbb V}F(\x)$,
and assume that $\mX^*$ is nonempty.  Throughout, we assume that $F$
satisfies the $\alpha$-quadratic-growth condition, for some $\alpha>0$,
\begin{align}\label{eq:qg-condition}
 F(\x)-F^*
 \ge \frac{\alpha}{2}\dist(\x,\mX^*)^2,
 \qquad \x\in\dom(\R).
\end{align}

We distinguish between the following two fundamental regimes.
\begin{itemize}
\item In the \emph{smooth case}, $f$ is differentiable and $\beta$-smooth with $\beta \geq \alpha$.
\item In the \emph{nonsmooth case}, no smoothness assumption is imposed on
$f$.  %
\end{itemize}

Standard first-order methods for Problem \eqref{eq:optProb}, which treat the function $\R$ in closed-form, rely on an oracle for $\R$ that can compute a minimizer of the proximal problem:
\begin{align}\label{eq:proxProb}
\min_{\z\in\mathbb{V}}\left\{\lambda\Vert{\z-\y}\Vert^2 + \R(\z)\right\},
\end{align}
for any $\lambda \geq 0$ and $\y\in\mathbb{V}$ \cite{beck2009fast, nesterov2013gradient, beck2017first}.

However, in many problems of interest, implementing this proximal oracle may be inefficient in high-dimensions. We discuss this in more detail in Section \ref{sec:wpo-examples}, however, a clean example to keep in mind is when $\mathbb{V}$ is a space of real matrices and $\R$ is a nuclear norm regularizer (sum of singular values) or an indicator for the positive semidefinite cone. In both scenarios, solving  \eqref{eq:proxProb} amounts in worst case to a full singular value decomposition (SVD) / eigen-decomposition, which is prohibitive when the dimension is very large. %

One important family of alternative methods is based on the Frank-Wolfe algorithm \cite{frank1956algorithm, Jaggi13}, which concerns the case that $\R$ is an indicator for a convex and compact set, and instead of solving the potentially expensive Problem \eqref{eq:proxProb} (which reduces to Euclidean projection in this case), requires only linear optimization steps over the feasible domain. However, these methods are well known to suffer from slow sublinear convergence rates even under quadratic growth or strong convexity conditions, as opposed to proximal methods that can leverage such conditions towards obtaining faster rates. A well known lower bound for these methods in case that $f$ is smooth and strongly convex is given in \cite{Jaggi13}. For completeness, we bring a simple lower bound covering the case that $f$ is nonsmooth and strongly convex (Theorem \ref{thm:lmo-nonsmooth-simplex-lb} in the appendix) \footnote {we note that while there is active research on Frank-Wolfe-type methods that can leverage quadratic growth / strong convexity conditions, such results most often come with additional dependencies not present in proximal methods and consider very specific structures of $\R$, and are beyond the scope of this work}.

In recent years several works proposed methods that build on an oracle that relaxes the exact minimization of the prox operator in \eqref{eq:proxProb} towards obtaining an oracle that can benefit from the quadratic growth condition, yet may be more efficient to implement in terms of runtime  \cite{allen2017linear, garber2021improved, garber2019fast, garber2023faster, ding2020k, garber2025weak, garber2025first}. The following definition, which is our main object of interest in this work, captures the most common definition. The term \textit{weak proximal oracle} was originally coined in \cite{garber2019fast}.

\begin{definition}[weak proximal oracle]\label{def:wpo}
A map $\mA:\mathbb{V}\times\reals_+\rightarrow\dom(\R)$ is called a weak
proximal oracle for Problem~\eqref{eq:optProb} if, for every
$\y\in\mathbb{V}$ and $\lambda>0$, it outputs a point $\x\in\dom(\R)$ such
that for any $\x^*\in\mX^*$,
\begin{align}\label{eq:wpo-def}
 \lambda\norm{\x-\y}^2+\R(\x)
 \le
 \lambda\norm{\x^*-\y}^2+\R(\x^*).
\end{align}
\end{definition}

Condition \eqref{eq:wpo-def} strongly relaxes the optimality condition of the standard prox operation \eqref{eq:proxProb}, and simply requires the output to be competitive with the set of minimizers of the global Problem \eqref{eq:optProb}. 

In Section \ref{sec:wpo-examples} we survey in detail settings in which this oracle can indeed be implemented with potentially much better runtime than the exact prox operator. For now,  we shall only mention that for our clean example from before  of a matrix  nuclear norm regularizer or an indicator for the positive semidefinite cone, in case all optimal solutions to the global Problem \eqref{eq:optProb} have low rank (which is indeed the case in many applications of interest), a WPO will require only a low-rank matrix decomposition, instead of a full-rank one for the exact prox operation.

In this work we ask:
\begin{center}
\textit{What are the complexities of first-order methods for Problem \eqref{eq:optProb} (under the assumptions listed above), when the function $\R$ is accessible only through a weak proximal oracle?}
\end{center}

To make this question even more focused, we note that \cite{allen2017linear, garber2021improved, garber2019fast, garber2023faster, ding2020k, garber2025weak, garber2025first} essentially showed that when $f$ is smooth, an order of $O\left({\frac{\beta}{\alpha}\log\frac{1}{\epsilon}}\right)$ gradient computations and WPO calls suffices to find an $\epsilon$-approximate solution (in function value). This matches the complexities of the proximal gradient method but with a potentially much more efficient oracle.  However, an important question remained open: whether accelerated rates (in the sense of Nesterov's accelerated gradient method), that replace the $\beta/\alpha$ factor with its square root $\sqrt{\beta/\alpha}$, are possible?

In case $f$ is nonsmooth, \cite{garber2019fast} showed that when $f$ can be replaced with a suitable smooth approximation, $\tilde{O}\left({\frac{1}{\alpha\epsilon}}\right)$ WPO calls suffice (here we let $\tilde{O}$ hide log factors and we focus only on dependence on $\alpha, \epsilon$), which matches (up to log factors and in terms of $\alpha,\epsilon$) the proximal subgradient method.
However, this approach is not generic (since it requires a known smooth approximation). The MOPES method \cite{MOPES}, in case $\R$ is an indicator for a convex and compact set, requires only an order of $O\left({1/\sqrt{\alpha\epsilon}}\right)$ exact projections onto the feasible set \footnote{while \cite{MOPES} did not consider a quadratic growth assumption and gave a $O(\Vert{\x_0-\x^*}\Vert/\epsilon)$ projection complexity bound, where $\x^*$ is some optimal solution, the reduction to the quadratic growth case follows from a standard restarting argument}. This raises the question whether WPO-based methods for the nonsmooth case can be accelerated, replacing the $1/(\alpha\epsilon)$ dependence with only $1/\sqrt{\alpha\epsilon}$?

Unfortunately, our main results are negative: WPO-based methods cannot be accelerated and the terms $\beta/\alpha$ and $1/(\alpha\epsilon)$ in the smooth and nonsmooth case, respectively, cannot be improved beyond perhaps log factors. Hence, while weak proximal oracles can substantially reduce the cost of individual oracle calls, this comes with an inherent loss in oracle complexity.

On the positive side, in case $f$ is nonsmooth, and without assuming $\R$ is an indicator for a convex and compact set, we provide a generic upper-bound requiring $\tilde{O}\left({\frac{1}{\alpha\epsilon}}\right)$ WPO and first-order oracle calls.

Together, these results essentially characterize the oracle complexity of weak proximal methods under quadratic growth. Weak proximal access can yield substantially cheaper iterations than exact proximal access, but unlike exact proximal access it does not support accelerated oracle complexity in the worst case.

Before we can present our results in more detail we require two additional definitions of oracles for accessing $\R$. Let us first present these definitions and then discuss their purpose.

\begin{definition}[restricted proximal oracle]\label{def:rpo}
Let $\widehat{\mK}\subseteq\dom(\R)$ be a closed, but not necessarily convex,
set satisfying $\mX^*\subseteq\widehat{\mK}$.  A map
$\mA:\mathbb{V}\times\reals_+\rightarrow\widehat{\mK}$ is called a restricted
proximal oracle associated with $\widehat{\mK}$ if, for every
$\y\in\mathbb{V}$ and $\lambda>0$, it outputs a point in
\begin{align}
 \argmin_{\z\in\widehat{\mK}}
 \left\{\lambda\norm{\z-\y}^2+\R(\z)\right\}.
\end{align}
Since $\mX^*\subseteq\widehat{\mK}$, every restricted proximal oracle is a
weak proximal oracle in the sense of Definition~\ref{def:wpo}.
\end{definition}

\begin{definition}[$k$-fold sparse projection oracle]\label{def:spo}
Suppose $\R$ is the indicator function of a convex and compact set $\mK\subset\mathbb{V}$ whose set of vertices is denoted $\mV$. Suppose every $\x^*\in\mX^*$ can be written as a convex combination of at most $s$ vertices from $\mV$, where $s<<n$. For a positive integer $k$ such that $ks \leq n$, a \textit{$k$-fold sparse projection oracle} is an oracle that for every $\y\in\mathbb{V}$, outputs a point in
\begin{align}
 \argmin_{\z\in\mC_{sk}}\norm{\z-\y},
\end{align}
where $\mC_{sk}\subseteq\mK$ contains all points in $\mK$ that can be written as the convex combination of at most $s\cdot{}k$ vertices from $\mV$.

In particular, this oracle implements the restricted proximal oracle of
Definition~\ref{def:rpo} with $\widehat{\mK}=\mC_{sk}$.
\end{definition}

Let us now discuss why we have three different oracle definitions. As already mentioned above, these oracles form a hierarchy:  the $k$-fold sparse projection oracle (in case $\R$ is an indicator for a set) implies a restricted proximal oracle, which in turn implies a weak proximal oracle. Our upper bounds will rely on the weakest assumption --- the weak proximal oracle, while our examples for concrete efficient implementations, discussed in Section \ref{sec:wpo-examples}, will all in fact implement the stronger restricted proximal oracle. Finally, all of our lower bounds will hold under the stronger assumption of a $1$-fold sparse projection oracle (in particular they all establish hard instances in which $\R$ is an indicator for a convex set). Moreover, our deterministic lower bounds (i.e., that hold against any deterministic method) will allow the strongest assumption of a $k$-fold sparse projection oracle for $k > 1$. These are provided to show that even significantly increasing the strength of the oracle for $\R$ by allowing $k >1$ cannot have (in worst case) a dramatic effect on the rate of convergence, and also does not enable acceleration for moderate values of $k$.  

Table~\ref{tab:main-results} summarizes our main oracle-complexity results together with the
corresponding known upper bounds.
The lower bounds already hold in the more restrictive setting $\R=I_{\mK}$, when the feasible
set is accessed through a sparse projection oracle and $f$ is strongly convex. In contrast, our
upper bounds apply to a general convex regularizer $\R$ and an arbitrary weak proximal oracle,
and require only quadratic growth of $F$; in particular, the optimal solution need not be unique.
For comparison, the table also reports the standard projection-based upper bounds: accelerated
projected gradient in the smooth case and restarted MOPES in the nonsmooth case. Thus the table
highlights the gap between the weaker oracle models considered here and access to an exact prox operator. For ease of presentation, the table presents the bounds in simplified form, omitting typically-lower-order terms. Refer to the theorems in the sequel for precise statements.

One thing that is important to note is that while the deterministic lower bounds with a $k$-fold sparse projection oracle reported in the table decay with $1/k$, our constructions of such oracles detailed in Section \ref{sec:wpo-examples} will have complexity that typically scales at least linearly with $k$. Thus, these lower bounds essentially establish that there is nothing to gain from taking $k>1$.

\begin{table}[tp]

\refstepcounter{table}
\label{tab:main-results}

\centering
\begin{minipage}{0.92\textwidth}

\small
\raggedright
\textbf{Table~\thetable:}
Summary of the main oracle-complexity results and exact-oracle benchmarks.
\par
\vspace{6pt}

\centering
\setlength{\tabcolsep}{4pt}
\renewcommand{\arraystretch}{1.18}

\begin{tabularx}{\linewidth}{
@{}
>{\raggedright\arraybackslash}p{0.15\linewidth}
>{\raggedright\arraybackslash}p{0.25\linewidth}
>{\raggedright\arraybackslash}p{0.27\linewidth}
>{\raggedright\arraybackslash}X
@{}
}
\toprule
Regime
& Result
& Oracle
& \#Oracle calls
\\
\midrule

Smooth
& Deterministic lower (\textbf{new})
& $k$-fold sparse projection
& $\displaystyle
  \Omega\!\left(\frac{\beta}{\alpha k}\right)$
\\[2pt]

Smooth
& Randomized lower (\textbf{new})
& $1$-fold sparse projection
& $\displaystyle
  \Omega\!\left(
    \frac{\beta/\alpha}{\log(\beta/\alpha)}
  \right)$
\\[2pt]

Smooth
& WPO upper (known)
& Weak proximal oracle
& $\displaystyle
  O\!\left(
    \frac{\beta}{\alpha}
    \log\frac{1}{\epsilon}
  \right)$
\\[2pt]

Smooth
& Exact proximal upper (known)
& Exact proximal oracle
& $\displaystyle
  O\!\left(
    \sqrt{\frac{\beta}{\alpha}}\,
    \log\!\left(
      \frac{1}{\epsilon}
    \right)
  \right)$
\\

\midrule

Nonsmooth
& Deterministic lower (\textbf{new})
& $k$-fold sparse projection
& $\displaystyle
  \Omega\!\left(
    \frac{G^2}{\alpha\epsilon k}
  \right)$
\\[2pt]

Nonsmooth
& Randomized lower (\textbf{new})
& $1$-fold sparse projection
& $\displaystyle
  \Omega\!\left(
    \frac{G^2}
    {\alpha\epsilon\,\log(G^2/(\alpha\epsilon))}
  \right)$
\\[2pt]

Nonsmooth
& WPO upper (\textbf{new})
& Weak proximal oracle
& $\displaystyle
  O\!\left(
    \frac{G^2}{\alpha\epsilon} 
  \right)$
\\[2pt]

Nonsmooth
& Exact projection upper (known)
& Exact Euclidean projection
& $\displaystyle
  O\!\left(
    \frac{G}{\sqrt{\alpha\epsilon}}
  \right)$
\\

\bottomrule
\end{tabularx}

\vspace{5pt}

\footnotesize
\raggedright
The lower bounds  allow arbitrarily many first-order objective calls.
The lower bounds hold already when $\R=I_{\mK}$ and $f$ is strongly convex,
whereas the WPO upper bounds require only quadratic growth and apply to a
general convex regularizer $\R$. $G$ is a local upper-bound on norms of subgradients.
\end{minipage}
\end{table}

\subsection{Representative implementations of weak proximal oracles}
\label{sec:wpo-examples}

We discuss three concrete representative implementations of the restricted proximal
oracle of Definition~\ref{def:rpo}, and hence of a WPO.   The examples show that restricting the proximal problem to a
structured set containing all optimal solutions can result in a much more efficient procedure than implementing the standard exact proximal oracle. The first two examples were already used for designing efficient WPO-based algorithms in \cite{allen2017linear, garber2021improved, garber2019fast, garber2023faster, ding2020k, garber2025weak, garber2025first}, however, here we discuss more general constructions. The third example was previously considered in \cite{garber2021frank}. These examples are detailed below and also summarized in Table \ref{tab:wpo-examples}.

\paragraph{Sparse solutions with a symmetric regularizer.}
First consider $\mathbb{V}=\reals^n$.  Suppose that $\R$ is invariant under
coordinate permutations and, in addition, one of the following holds:
\begin{enumerate}
\item $\dom(\R)\subseteq\reals_+^n$; or
\item $\R$ is invariant under coordinate-wise sign changes.
\end{enumerate}
This setting includes, for example, symmetric norm
regularizers, indicators of permutation- and sign-symmetric sets, and sums
of such terms. %

Suppose that every $\x^*\in\mX^*$ satisfies
$\nnz(\x^*)\le s$, and for some $r\ge s$ define
\[
 \mC_r
 :=
 \left\{
 \x\in\dom(\R):\nnz(\x)\le r
 \right\}.
\]
The proximal problem over $\mC_r$ can be reduced to only $r$ coordinates.
For $T\subseteq[n]$, let $\I_T$ denote the matrix formed by the columns of
the identity matrix indexed by $T$, and let $\y_T$ denote the restriction
of $\y$ to $T$.

\begin{lemma}%
\label{lem:symmetric-sparse-projection}
Let $\R$ satisfy the symmetry conditions above, and let $s\le r\le n$.
Fix $\y\in\reals^n$ and $\lambda>0$.

If $\dom(\R)\subseteq\reals_+^n$, let $T$ contain the indices of the $r$
largest entries of $\y$.  If $\R$ is invariant under coordinate-wise sign
changes, let $T$ contain the indices of the $r$ largest entries of $\y$ in
absolute value.  Let
\[
 \widehat{\u}
 \in
 \argmin_{\u\in\reals^r}
 \left\{
 \lambda\norm{\u-\y_T}^2+\R(\I_T\u)
 \right\},
 \qquad
 \widehat{\x}:=\I_T\widehat{\u}.
\]
Then
\[
 \widehat{\x}
 \in
 \argmin_{\x\in\mC_r}
 \left\{
 \lambda\norm{\x-\y}^2+\R(\x)
 \right\}.
\]
Consequently, if every $\x^*\in\mX^*$ is $s$-sparse, then for every
$r\ge s$ this construction implements a weak proximal oracle.
\end{lemma}

Thus the WPO requires only finding the $r$ largest coordinates and solving
an $r$-dimensional proximal problem.  For an indicator regularizer this
reduces to the sparse restricted projection considered previously in
\cite{beck2016minimization}.  More generally, it applies directly to
non-indicator regularizers with the same symmetry.

When $r\ll n$, the computational savings need not come only from the
proximal operation itself.  The sparse oracle output may also make the
first-order computations substantially cheaper.  For example, consider the
dense quadratic
\[
 f(\x)=\frac{1}{2}\x^\top\A\x+\b^\top\x.
\]
If $\A\x$ is maintained and an update has the form
$\x^+=(1-\gamma)\x+\gamma\v$, where $\v$ is $r$-sparse, then
$\A\x^+=(1-\gamma)\A\x+\gamma\A\v$,
and hence the new gradient can be computed in $O(nr)$ time, which is considerably faster than the
$O(n^2)$ time required for a fresh multiplication by a dense $\A$, in case $r \ll n$.

A related setting is Bayesian D-optimal design,
\[
 f(\x)=-\log\det\M(\x),
 \qquad
 \M(\x):=\M_0+\sum_{i=1}^n x_i\a_i\a_i^\top,
 \qquad \M_0\succ0,
\]
where $\a_i\in\reals^d$ and, for example, $\x$ belongs to the unit simplex. Here updating the gradient after a dense update to $\x$ (as occurs in worst case in proximal gradient methods) requires expensive matrix inversion using $O(d^3)$ time. However, when only a few entries in $\x$ are updated (as discussed for the quadratic case above), the inverse could be updated much more efficiently via the Sherman–Morrison–Woodbury formula.

\paragraph{Low-rank solutions with a spectral regularizer.}
There are two direct matrix analogues of the construction above.  First, let
$q=\min\{d_1,d_2\}$ and consider
$\mathbb{V}=\reals^{d_1\times d_2}$ endowed with the Frobenius norm, whose
ambient dimension is $n=d_1d_2$.  Suppose
that $\R$ is unitarily invariant, that is,
$\R(\P\X\Q^\top)=\R(\X)$
for every pair of orthogonal matrices $\P$ and $\Q$ of the appropriate
dimensions.  Equivalently, there is a proper, closed, convex, and
permutation-invariant function
$\rho:\reals_+^q\rightarrow\reals\cup\{+\infty\}$ such that
$\R(\X)=\rho\bigl(\boldsymbol{\sigma}(\X)\bigr)$,
where $\boldsymbol{\sigma}(\X)$ denotes the vector of singular values of
$\X$.  Thus $\R(\X)$ depends on $\X$ only through its singular values.  This
includes, for example, nuclear- and Schatten-norm regularizers, indicators
of spectrally symmetric sets, and sums of such terms.

There is also a symmetric-matrix version.  In this case,
$\mathbb{V}=\mbS^d$ is endowed with the Frobenius norm,
$\dom(\R)\subseteq\mbS_+^d$, and $\R$ is invariant under orthogonal
conjugation, that is,
$\R(\U\X\U^\top)=\R(\X)$
for every orthogonal matrix $\U$.  Equivalently, there is a proper, closed,
convex, and permutation-invariant function
$\rho:\reals^d\rightarrow\reals\cup\{+\infty\}$ with
$\dom(\rho)\subseteq\reals_+^d$ such that
$\R(\X)=\rho\bigl(\boldsymbol{\lambda}(\X)\bigr)$,
where $\boldsymbol{\lambda}(\X)$ denotes the vector of eigenvalues of
$\X$.  This setting includes, in particular, the indicator of the PSD cone.

In either setting, suppose that every $\X^*\in\mX^*$ has rank at most $r$,
and define
\[
 \mC_r
 :=
 \left\{
 \X\in\dom(\R):\rank(\X)\le r
 \right\}.
\]
As in the sparse-vector case, the proximal problem over $\mC_r$ reduces to
an $r$-dimensional problem.

\begin{lemma}%
\label{lem:spectral-sparse-projection}
In the rectangular-matrix setting, suppose
$\R(\X)=\rho\bigl(\boldsymbol{\sigma}(\X)\bigr)$,
where $\rho:\reals_+^q\rightarrow\reals\cup\{+\infty\}$ is permutation
invariant.  Fix $\Y\in\reals^{d_1\times d_2}$ and $\lambda>0$, and let
$\Y
 =
 \U\operatorname{Diag}(\boldsymbol{\sigma})\V^\top$,
 $\sigma_1\ge\cdots\ge\sigma_q\ge0$, be a SVD, where $\U$ and $\V$ have $q$ orthonormal
columns. Let $\U_r$ and $\V_r$ contain the
first $r$ columns of $\U$ and $\V$, respectively.  Let
\[
 \widehat{\u}
 \in
 \argmin_{\u\in\reals_+^r}
 \left\{
 \lambda\norm{\u-\boldsymbol{\sigma}_{1:r}}^2
 +
 \rho\!\left((\u,\mathbf{0}_{q-r})\right)
 \right\},
\]
and set
$\widehat{\X}
 :=
 \U_r\operatorname{Diag}(\widehat{\u})\V_r^\top$.

In the symmetric-matrix setting, suppose
$\R(\X)=\rho\bigl(\boldsymbol{\lambda}(\X)\bigr)$,
where $\rho:\reals^d\rightarrow\reals\cup\{+\infty\}$ is permutation
invariant and $\dom(\rho)\subseteq\reals_+^d$.  Fix
$\Y\in\mbS^d$ and $\lambda>0$, and let
$ \Y
 =
 \U\operatorname{Diag}(\boldsymbol{\gamma})\U^\top$,
 $\gamma_1\ge\cdots\ge\gamma_d$,
be an eigendecomposition.  Let $\U_r$ contain the first $r$ columns of
$\U$.  Let
\[
 \widehat{\u}
 \in
 \argmin_{\u\in\reals_+^r}
 \left\{
 \lambda\norm{\u-\boldsymbol{\gamma}_{1:r}}^2
 +
 \rho\!\left((\u,\mathbf{0}_{d-r})\right)
 \right\},
\]
and set
$\widehat{\X}
 :=
 \U_r\operatorname{Diag}(\widehat{\u})\U_r^\top$.

In either setting,
\[
 \widehat{\X}
 \in
 \argmin_{\X\in\mC_r}
 \left\{
 \lambda\norm{\X-\Y}_F^2+\R(\X)
 \right\}.
\]
Consequently, if every $\X^*\in\mX^*$ has rank at most $r$, the
corresponding construction implements a weak proximal oracle.
\end{lemma}

Two standard special cases make the computational saving particularly
concrete.  If $\R(\X)=\tau\norm{\X}_*$, where $\norm{\cdot}_*$ denotes the matrix nuclear norm (sum of singular values), the rank-$r$ restricted proximal
oracle computes the leading $r$ singular components of $\Y$ and
soft-thresholds their singular values at $\tau/(2\lambda)$, whereas the exact
proximal map soft-thresholds the entire singular spectrum, which in worst
case requires a full SVD.
In the symmetric-matrix setting, if $\R=I_{\mbS_+^d}$, i.e., $\R$ is the indicator for the positive semidefinite cone, the rank-$r$
restricted oracle requires the largest $r$ eigenpairs of $\Y$, whereas
exact projection onto the PSD cone requires all eigenpairs and hence a full eigen-decomposition. 

Note also that, as in the previous sparse-vector setting, for suitable
objective functions these low-rank updates can also significantly reduce the
cost of updating the gradient from one iteration to the next.

\paragraph{Vertex solutions over $0/1$ polytopes.}
For our third example suppose
$\R=I_{\mK}$, $\mK=\conv(\mV)$, $\mV\subseteq\{0,1\}^n$,
and suppose that every optimal solution is a vertex,
$\mX^*\subseteq\mV$.  Taking $\widehat{\mK}=\mV$ in
Definition~\ref{def:rpo}, the restricted proximal oracle reduces,
independently of $\lambda$, to
\[
 \argmin_{\v\in\mV}\norm{\v-\y}^2
 =
 \argmin_{\v\in\mV}
 \ip{\mathbf{1}-2\y}{\v},
\]
where we used
$\norm{\v}^2=\mathbf{1}^\top\v$ for every $\v\in\{0,1\}^n$.
Thus projection onto the vertices is exactly one linear-optimization call
over the polytope. This was already observed in \cite{garber2021frank}.
There are many nontrivial polytopes for which this linear problem is much
simpler than Euclidean projection onto the full polytope
\cite{Jaggi13,combettes2021}.  For a matroid base polytope it is solved
by the greedy algorithm; for the spanning-tree polytope it is a
minimum-weight spanning-tree computation; for the Birkhoff polytope it is
a linear assignment problem; and for an $s$--$t$ path polytope on an
acyclic network it is a shortest-path computation.  Exact Euclidean
projection, which in this case is the exact proximal operation, instead
requires solving the corresponding convex quadratic optimization problem
over the entire polytope.

\begin{table}[tp]
\centering
\begin{minipage}{0.94\textwidth}
\centering
\small
\caption{Representative implementations of weak proximal oracles and
comparison with exact proximal operations.}
\label{tab:wpo-examples}
\setlength{\tabcolsep}{5pt}
\renewcommand{\arraystretch}{1.2}

\begin{tabularx}{\linewidth}{
 @{}
 >{\raggedright\arraybackslash}p{0.21\linewidth}
 >{\raggedright\arraybackslash}X
 >{\raggedright\arraybackslash}p{0.29\linewidth}
 @{}
}
\toprule
Structure
& Weak proximal oracle
& Exact proximal oracle
\\
\midrule

$s$-sparse optimizers; symmetric regularizer
&
Select $r$ largest coordinates and solve resulting
$r$-dimensional proximal problem;  output is $r$-sparse
&
Full $n$-dimensional proximal problem;  output is dense in  worst
case
\\

\addlinespace

Rank-$r$ optimizers; unitarily invariant regularizer
&
Compute rank-$r$ truncated SVD and solve $r$-dimensional proximal
problem for retained singular values; roughly $O(rd^2)$ time for 
dense $d\times d$ matrix
&
Full spectral proximal problem; in worst case,  full SVD and
$O(d^3)$ time for  dense $d\times d$ matrix
\\

\addlinespace

Vertex optimizers over a $0/1$ polytope
&
Solve one linear optimization problem over polytope using, e.g.,
greedy optimization, a minimum spanning tree, linear assignment, or a
shortest-path computation
&
Convex quadratic optimization over  full polytope
\\

\bottomrule
\end{tabularx}
\end{minipage}
\end{table}

\subsection{Organization of the paper}
The remainder of the paper is organized as follows.
Section~\ref{sec:deterministic-lower-bounds} develops a resisting-oracle
construction and uses it to establish deterministic lower bounds for the
nonsmooth and smooth regimes under a $k$-fold sparse projection oracle
(Theorems~\ref{thm:nonsmooth-sparse-lb} and~\ref{thm:smooth-sparse-lb}).
Section~\ref{sec:randomized-lower-bounds} extends these lower bounds to
randomized methods for the $1$-fold sparse projection oracle, at the cost of
a logarithmic factor
(Theorems~\ref{thm:simple-randomized-nonsmooth-lb}
and~\ref{thm:simple-randomized-smooth-lb}).
Section~\ref{sec:qg-upper-bounds} complements these results with WPO-based
upper bounds under quadratic growth, first recalling the smooth-case result
and then developing a new nonsmooth method and its analysis
(Theorems~\ref{thm:qg-smooth-wpo-upper}
and~\ref{thm:qg-nonsmooth-wpo-upper}).

\section{Deterministic Lower Bounds}\label{sec:deterministic-lower-bounds}
We shall assume throughout this section that $\R$ is an  indicator for a convex and compact set and that $f$ is strongly convex, which in particular means that the optimal solution is unique. Most importantly, the lower bounds will hold under a stronger access model to $\R$ than a weak proximal oracle; we shall assume that the feasible set corresponding to $\dom(\R)$ is accessible through the stronger $k$-fold sparse projection oracle (Definition \ref{def:spo}). Our lower bounds will hold already for the case that the (unique) optimal solution is a vertex of the feasible set (i.e., $s=1$ in terms of the oracle definition).

The convex and compact set, for which $\R$ is the indicator, will be denoted $\mK_{N,R}:=\conv(\mV_{N,R})$, where $\mV_{N,R}$ is the set of vertices. Since our construction is such that the optimal solution is a vertex in $\mV_{N,R}$,  the set corresponding to the $k$-fold sparse projection operation is
\begin{align}\label{eq:detLB:Ck}
 \mC_k=\conv_k(\mV_{N,R})
 :=\left\{\sum_{i=1}^p\lambda_i\v_i:
 p\le k,\ \v_i\in\mV_{N,R},\ \lambda_i\ge0,\ \sum_i\lambda_i=1\right\}.
\end{align}

We use the standard black-box oracle model: the method is not given an
explicit description of $\mK_{N,R}$, and its only access to the feasible set is through projection onto
$\mC_k$.  The method may otherwise perform arbitrary computations and make
arbitrarily many first-order calls to the objective.

The proof uses a deterministic resisting oracle.  We fix an arbitrary
deterministic method, but leave some of the orthonormal vectors defining the
feasible set unspecified while the method interacts with the sparse-projection
oracle.  At each $k$-fold sparse-projection call, the not-yet-specified vectors
are chosen so that an exact projection can be returned without using the
target vertex $\s^*$, while fixing at most $k$ new vectors $\q_j$.  The
objective does not depend on these vectors, so arbitrarily many first-order
calls to the objective provide no information about the choices that remain.
After all sparse-projection calls and the final output are known, the
remaining orthonormal vectors are chosen so that every previously returned
point is still an exact projection w.r.t. $\mC_k$, while the output stays a constant distance from the target
vertex.  Since at most $k$ new vectors $\q_j$ are fixed per sparse-projection
call, this yields an $\Omega(N/k)$ geometric lower bound, which we then
combine with nonsmooth and smooth objectives.

\subsection{The resisting construction}

Let $N\ge2$, $m=2N$, and let
$\q_1,\ldots,\q_N\in\reals^m$ be orthonormal.  Define
\[
 \u_N
 :=
 \frac1{\sqrt N}\sum_{i=1}^N\q_i,
 \qquad
 \v_j
 :=
 \frac1{\sqrt N}\sum_{i=1}^j\q_i,
 \qquad
 \rho_j
 :=
 \frac{j}{N},
 \quad j\in[N].
\]
Then
\begin{align}\label{eq:surrogate-identities}
 \norm{\v_j}^2
 &=
 \rho_j,
 \quad
 \rho_1\le\cdots\le\rho_N=1, \nonumber \\
\ip{\v_i}{\v_j}
 &=
 \ip{\v_i}{\u_N}
 =
 \rho_i,
 \quad i<j\le N.
\end{align}
Thus, as seen from every earlier $\v_i$, a later $\v_j$ is
indistinguishable from $\u_N$.

Let $\rr\in\reals^m$ be a unit vector, and write a point of
$\reals^{3m}$ as $(\x,\z,\h)$ with
$\x,\z,\h\in\reals^m$.  For $R>0$ define
\begin{align}\label{eq:deterministic-vertices}
 \s^*
 &:=
 R\left(
   \frac{\u_N}{\sqrt2},
   0,
   \frac{\rr}{\sqrt2}
 \right),
 \qquad
 \s_j
 :=
 R\left(
   \frac{\v_j}{\sqrt2},
   \frac{\q_j}{\sqrt2},
   0
 \right),
 \quad j\in[N].
\end{align}
Let
$\mV_{N,R}:=\{\s^*,\s_1,\ldots,\s_N\}$ and
$\mK_{N,R}:=\conv(\mV_{N,R})$.
We call $\s^*$ the target vertex and
$\s_1,\ldots,\s_N$ the non-target vertices.

Each point in $\mV_{N,R}$ is a vertex of $\mK_{N,R}$.
Indeed, $\s_j$ is the unique maximizer over $\mV_{N,R}$ of the
linear functional $(\x,\z,\h)\mapsto\ip{\q_j}{\z}$,
while $\s^*$ is the unique maximizer of
$(\x,\z,\h)\mapsto\ip{\rr}{\h}$.
Moreover, $\norm{\s^*}=R$,
 $\norm{\s_j}^2
 =
 \frac{R^2}{2}(1+\rho_j)
 \le R^2$,
so $\mK_{N,R}$ is contained in the Euclidean ball of radius $R$. Also note that
\begin{align}\label{eq:deterministic-gram}
\forall i < j: \qquad \ip{\s_i}{\s_j}
 =
 \frac{R^2}{2}\rho_i
 =
 \ip{\s_i}{\s^*}.
\end{align}

\begin{theorem}[Deterministic geometric lower bound]
\label{thm:deterministic-geometric-hiding}
Fix an integer $N\ge5$, $1\le k\le6N$, and $R>0$.  For every
deterministic method making at most
$T\le\lfloor N/(4k)\rfloor$ projections onto $\mC_k$ (see \eqref{eq:detLB:Ck}), there exist
orthonormal vectors $\q_1,\ldots,\q_N\in\reals^{2N}$, a unit vector
$\rr\in\reals^{2N}$, and a deterministic tie-breaking rule for the
projection oracle such that, for the resulting set
$\mK_{N,R}\subset\reals^{6N}$ defined
by \eqref{eq:deterministic-vertices}, the method returns either an
infeasible point or a point $\widehat\w$ satisfying
\begin{align*}%
 \norm{\widehat\w-\s^*}^2
 \ge
 \frac{R^2}{2}.
\end{align*}
Between sparse-projection calls, the method may also receive arbitrary
answers from auxiliary oracles whose response rules do not depend on
$\q_1,\ldots,\q_N$ or $\rr$.  In the applications below, these
auxiliary answers are precisely the free first-order oracle
answers.
\end{theorem}

\begin{proof}
We construct the vectors gradually while interacting with the method.
Suppose that, immediately before a new sparse-projection call,
$\q_1,\ldots,\q_\ell$ have already been fixed, and write the input to the sparse projection oracle call as
$y=(\p,\z,\h)$, $\p,\z,\h\in\reals^m$.
Provisionally complete $\q_1,\ldots,\q_\ell$ to an orthonormal family
$\q_1,\ldots,\q_N$ such that every $\q_j$, $j>\ell$, is orthogonal
to the $\p$- and $\z$-blocks of all sparse-projection inputs seen so
far, including the current call.  Likewise, choose $\rr$ orthogonal
to the $\h$-blocks of all these inputs.  Only the vectors needed for
the current oracle answer will be fixed permanently.

For every $j>\ell$, the orthogonality conditions give
\begin{align}\label{eq:deterministic-hidden-score}
 \ip{\y}{\s_j}
 &=
 \frac{R}{\sqrt2}
 \left(
   \ip{\p}{\v_j}
   +
   \ip{\z}{\q_j}
 \right)
 =
 \frac{R}{\sqrt{2N}}
 \sum_{i=1}^{\ell}\ip{\p}{\q_i}
 =
 \ip{\y}{\s^*}.
\end{align}
Thus, all not-yet-fixed non-target vertices and the target vertex have
the same inner product with the current sparse-projection input.
Choose any $\w\in\argmin_{\x\in\mC_k}\norm{\x-\y}^2$
for the provisional completion, and fix a representation of $\w$
using at most $k$ vertices.  Write it as
\begin{align}\label{eq:original-projection-decomposition}
 \w
 =
 \sum_{i\in I}\lambda_i\s_i
 +
 \sum_{p=1}^q\mu_p\widetilde{\s}_p,
 \qquad
 I\subseteq[\ell],
\end{align}
where $\lambda_i,\mu_p\ge0$, their sum is one, and
$\widetilde{\s}_1,\ldots,\widetilde{\s}_q$ are the vertices in this
representation that are not already fixed.  Associate
$\widetilde{\s}_p=\s_j$ with the index $j$, and associate
$\widetilde{\s}_p=\s^*$ with the index $+\infty$.  Order them so
that their associated indices satisfy $j_1<\cdots<j_q$.
In particular, the target is last whenever it is present.

Replace these vertices, with the same coefficients, by the next
unused non-target vertices:
\begin{align}\label{eq:replacement-projection-decomposition}
 \widetilde{\w}
 :=
 \sum_{i\in I}\lambda_i\s_i
 +
 \sum_{p=1}^q\mu_p\s_{\ell+p}.
\end{align}
Since $|I|+q\le k$, we have $\widetilde{\w}\in\mC_k$.  Moreover, on
the $t$th sparse-projection call,
\[
 \ell+q
 \le
 kt
 \le
 kT
 \le
 \frac{N}{4},
\]
so all the replacement vertices are available.

By \eqref{eq:deterministic-hidden-score} we have that, $\ip{\y}{\widetilde{\w}}
 =
 \ip{\y}{\w}$.
We next show that
$\norm{\widetilde{\w}}\le\norm{\w}$.
For every already fixed vertex $\s_i$, $i\in I$, and every replaced
vertex, the corresponding cross inner product is unchanged by
\eqref{eq:deterministic-gram}:
\[
 \ip{\s_i}{\widetilde{\s}_p}
 =
 \frac{R^2}{2}\rho_i
 =
 \ip{\s_i}{\s_{\ell+p}}.
\]

For each $p$, if $\widetilde{\s}_p=\s_{j_p}$ is a non-target vertex,
then $j_p\ge\ell+p$ and,
\[
 \norm{\s_{\ell+p}}^2
 =
 \frac{R^2}{2}(1+\rho_{\ell+p})
 \le
 \frac{R^2}{2}(1+\rho_{j_p})
 =
 \norm{\widetilde{\s}_p}^2.
\]
If $\widetilde{\s}_p=\s^*$, the same inequality follows from
$\norm{\s_{\ell+p}}\le R=\norm{\s^*}$.

Finally, for $1\le p<r\le q$, the vertex
$\widetilde{\s}_p$ cannot be the target, and therefore $j_p$ is
finite.  Since $j_p\ge\ell+p$, we have using \eqref{eq:deterministic-gram}
\[
 \ip{\s_{\ell+p}}{\s_{\ell+r}}
 =
 \frac{R^2}{2}\rho_{\ell+p}
 \le
 \frac{R^2}{2}\rho_{j_p}
 =
 \ip{\widetilde{\s}_p}{\widetilde{\s}_r}.
\]
Since all coefficients in~\eqref{eq:original-projection-decomposition}
are nonnegative, the preceding comparisons imply that $\norm{\widetilde{\w}}^2
 \le
 \norm{\w}^2$.
Together with
$\ip{\y}{\widetilde{\w}}=\ip{\y}{\w}$, this gives
$\norm{\widetilde{\w}-\y}^2
 \le
 \norm{\w-\y}^2$.
Since $\w$ is a minimizer over $\mC_k$ and
$\widetilde{\w}\in\mC_k$, equality must hold.  Hence
$\widetilde{\w}$ is also an exact projection of $\y$ onto $\mC_k$.

Return $\widetilde{\w}$ and permanently fix 
$\q_{\ell+1},\ldots,\q_{\ell+q}$.  Thus the returned point does not
use the target vertex and at most $k$ new vectors are fixed.  If an
input is repeated, return exactly the same answer as before.

Every returned vertex depends only on vectors that have already been
fixed.  Moreover, for every previous sparse-projection input, its
inner products with all not-yet-fixed vertices are given by
\eqref{eq:deterministic-hidden-score}, while all vertex norms and
pairwise inner products are determined by
\eqref{eq:surrogate-identities}--\eqref{eq:deterministic-gram}.
Thus changing the still-unfixed orthonormal vectors does not change
the value of any previous sparse-projection problem.  Consequently,
all previous answers remain exact projections as the construction
continues, and the answers are consistent with a single deterministic
tie-breaking rule.  Calls to the auxiliary oracles do not affect the
construction because their response rules are independent of
$\q_1,\ldots,\q_N$ and $\rr$.

After all $T$ sparse-projection calls, at most $\ell\le kT$
vectors have been fixed.  The remaining $N-\ell$ vectors may be
chosen orthonormal and orthogonal to the $\p$- and $\z$-blocks of all
$T$ sparse-projection inputs, since the relevant orthogonal complement
has dimension at least
\[
 m-\ell-2T
 =
 2N-\ell-2T
 \ge
 N-\ell,
\]
where we used $T\le N/4$.  Finally, choose the unit vector $\rr$
orthogonal to the $\h$-blocks of all sparse-projection inputs and to
the third block $\widehat{\h}$ of the final output; this is possible
since these impose at most $T+1<m$ linear constraints.

Suppose now that $\widehat{\w}$ is feasible, and write
\[
 \widehat{\w}
 =
 \theta^*\s^*
 +
 \sum_{j=1}^N\theta_j\s_j,
 \qquad
 \theta^*,\theta_j\ge0,
 \qquad
 \theta^*+\sum_{j=1}^N\theta_j=1.
\]
Comparing the third blocks gives $\widehat{\h}
 =
 \frac{\theta^*R}{\sqrt2}\rr$.
Since $\rr\perp\widehat{\h}$, we must have
$\theta^*=0$, and therefore $\widehat{\h}=0$.  This implies that,
\[
 \norm{\widehat{\w}-\s^*}^2
 \ge
 \norm{\widehat{\h}-R\rr/\sqrt2}^2
 =
 \frac{R^2}{2},
\]
which proves the theorem.
\end{proof}

\subsection{The deterministic lower bounds}

\begin{theorem}[The nonsmooth case]
\label{thm:nonsmooth-sparse-lb}
Fix $\alpha,G>0$, an integer $N\ge5$, and $1\le k\le6N$.  For every
deterministic method using at most $\lfloor N/(4k)\rfloor$ calls to a
$k$-fold sparse projection oracle, there exists an instance in
dimension $n=6N$ with compact feasible set, an
$\alpha$-strongly convex objective whose feasible subgradients have
norm at most $G$, and a unique vertex minimizer, such that the
returned point $\widehat\w$ is either infeasible or
\begin{align}\label{eq:nonsmooth-deterministic-gap}
 f(\widehat\w)-f(\w^*)
 \ge
 \frac{G^2}{36\alpha N}.
\end{align}
Arbitrarily many first-order calls between sparse projections are
allowed.

Consequently, there is a universal constant $c_0>0$ such that,
whenever $G^2/(\alpha\epsilon{}k)$ is sufficiently large, every
deterministic method that guarantees an $\epsilon$-optimal feasible
point on all such instances requires at least
$c_0G^2/(\alpha\epsilon k)$ $k$-fold sparse projections in the worst
case.
\end{theorem}

\begin{proof}
Fix a deterministic method.  Use the preceding geometry with
$R
 =
 G/(3\alpha\sqrt N)$
and $\w^*=\s^*$, and define
\begin{align}\label{eq:nonsmooth-deterministic-objective}
 f(\x,\z,\h)
 &:=
 \frac\alpha2
 \bigl(
   \norm\x^2+\norm\z^2+\norm\h^2
 \bigr)
 +
 \frac{2G}{3}\norm\z,
 \qquad
 \x,\z,\h\in\reals^m.
\end{align}
This function is $\alpha$-strongly convex.  Since
$\mK_{N,R}$ is contained in the ball of radius $R$, every feasible
subgradient has norm at most
\[
 \alpha R+\frac{2G}{3}
 =
 \frac{G}{3\sqrt N}+\frac{2G}{3}
 \le
 G.
\]
To specify the free first-order oracle completely, let it return
\[
 \left(
   \alpha\x,\,
   \alpha\z+\frac{2G}{3}\frac{\z}{\norm\z},\,
   \alpha\h
 \right)
\]
when $\z\ne0$, and $(\alpha\x,0,\alpha\h)$ when $\z=0$.
In particular, these answers depend only on the first-order input
$(\x,\z,\h)$, and not on the vectors $\q_j$ or $\rr$ defining the
feasible set.

In order to verify the optimality of $\w^*=\s^*$, and although the first-order oracle returns the zero element of the
$\z$-subdifferential when $\z=0$, for the optimality certificate we may
choose
\[
 \g^*
 =
 \left(
   \frac{\alpha R}{\sqrt2}\u_N,\,
   \frac{2G}{3}\u_N,\,
   \frac{\alpha R}{\sqrt2}\rr
 \right)
 \in\partial f(\w^*),
\]
since $\norm{\u_N}=1$.

For every non-target vertex $\s_j$,
\[
 \s_j-\s^*
 =
 R\left(
   \frac{\v_j-\u_N}{\sqrt2},
   \frac{\q_j}{\sqrt2},
   -\frac{\rr}{\sqrt2}
 \right),
\]
and therefore
\[
 \begin{aligned}
 \ip{\g^*}{\s_j-\s^*}
 &=
 \frac{\alpha R^2}{2}\ip{\u_N}{\v_j-\u_N}
 +
 \frac{2GR}{3\sqrt2}\ip{\u_N}{\q_j}
 -
 \frac{\alpha R^2}{2}\norm{\rr}^2\\
 &=
 \frac{\alpha R^2}{2}(\rho_j-1)
 +
 \frac{2GR}{3\sqrt{2N}}
 -
 \frac{\alpha R^2}{2}\\
 &=
 \frac{G^2}{9\alpha N}
 \left(
   \frac{\rho_j}{2}-1+\sqrt2
 \right) > 0,
 \end{aligned}
\]
where we used
$\norm{\u_N}=\norm{\rr}=1$,
$\ip{\u_N}{\v_j}=\rho_j$,
$\ip{\u_N}{\q_j}=1/\sqrt N$, and
$R=G/(3\alpha\sqrt N)$.

Thus, $\w^*=\s^*$ is indeed the unique constrained minimizer.  Theorem
\ref{thm:deterministic-geometric-hiding} and strong convexity now
give
\[
 f(\widehat\w)-f(\w^*)
 \ge
 \frac{\alpha R^2}{4}
 =
 \frac{G^2}{36\alpha N},
\]
which proves the theorem.
\end{proof}

\begin{remark}
Note that the bounded-subgradient property in
Theorem~\ref{thm:nonsmooth-sparse-lb} holds on more than just the
feasible set.  Indeed,
$\mK_{N,R}\subseteq\{\w:\norm{\w}\le R\}$, and for every
$\norm{\w}\le R$ and every $\g\in\partial f(\w)$,
\[
 \norm{\g}
 \le
 \alpha R+\frac{2G}{3}
 =
 \frac{G}{3\sqrt N}+\frac{2G}{3}
 \le G.
\]
Thus the lower bound continues to hold when subgradients are assumed
bounded on a Euclidean ball containing the feasible set.  This is
the same type of neighborhood-access assumption used by
MOPES~\cite{MOPES} to obtain its improved projection complexity for
nonsmooth constrained optimization.
\end{remark}

\begin{theorem}[The smooth case]
\label{thm:smooth-sparse-lb}
Fix $\alpha,\beta>0$, let $\kappa:=\beta/\alpha$,
$N
 :=
 \left\lfloor(\kappa-5)/16\right\rfloor$,
and suppose that $\kappa\ge85$.  For every $1\le k\le6N$ and every
deterministic method using at most $\lfloor N/(4k)\rfloor$ calls to a
$k$-fold sparse projection oracle, there exists an instance in
dimension $n=6N$ with an $\alpha$-strongly convex and
$\beta$-smooth quadratic objective, a unique vertex minimizer, and a
feasible set diameter $D$ satisfying $D\le2$, such that the
returned point $\widehat{\w}$ is either infeasible or
\begin{align}\label{eq:smooth-deterministic-gap}
 f(\widehat\w)-f(\w^*)
 \ge
 \frac{\alpha}{4}.
\end{align}
Consequently, for $\epsilon=\alpha/8$, every deterministic method
that guarantees an $\epsilon$-optimal feasible point on all such
instances requires
$\Omega\left(\beta/(\alpha k)\right)$ calls to the $k$-fold sparse
projection oracle in worst case.
\end{theorem}

\begin{proof}
Fix a deterministic method.  Since $\kappa\ge85$, we have $N\ge5$.
Use the same geometry with $R=1$, set $\w^*=\s^*$, and define $\gamma
 :=
 2\sqrt{\kappa-5}/\kappa$.
Consider the quadratic objective
\begin{align}\label{eq:smooth-deterministic-objective}
 f(\x,\z,\h)
 &:=
 \frac{\beta}{2}
 \left(
   \frac{5}{\kappa}\norm\x^2
   +
   2\gamma\ip\x\z
   +
   \left(1-\frac4\kappa\right)\norm\z^2
   +
   \frac1\kappa\norm\h^2
 \right),
 ~
 \x,\z,\h\in\reals^m.
\end{align}
Its gradient is
\[
 \nabla f(\x,\z,\h)
 =
 \beta
 \left(
   \frac5\kappa\x+\gamma\z,\,
   \gamma\x+\left(1-\frac4\kappa\right)\z,\,
   \frac1\kappa\h
 \right).
\]
In particular, the objective and its gradient depend only on the
query $(\x,\z,\h)$ and not on the vectors $\q_j$ or $\rr$ defining
the feasible set and its target vertex.  Hence first-order
objective calls reveal no information about these hidden vectors.

The Hessian with respect to the three blocks $(\x,\z,\h)$ is
\[
 \nabla^2 f
 =
 \beta
 \begin{pmatrix}
  \frac5\kappa\Id_m
  &
  \gamma\Id_m
  &
  0
  \\
  \gamma\Id_m
  &
  \left(1-\frac4\kappa\right)\Id_m
  &
  0
  \\
  0
  &
  0
  &
  \frac1\kappa\Id_m
 \end{pmatrix}.
\]
The trace and determinant of the upper-left $2\times2$ block,
without the factor $\beta$, are respectively
\[
 1+\frac1\kappa
 \qquad\text{and}\qquad
 \frac5\kappa
 \left(
   1-\frac4\kappa
 \right)
 -
 \gamma^2
 =
 \frac1\kappa.
\]
Its two distinct eigenvalues are therefore $1$ and $1/\kappa$.  Together with
the eigenvalue $1/\kappa$ of the third block, this shows that the
exact strong-convexity and smoothness parameters are indeed 
$\beta/\kappa
 =
 \alpha$ and $\beta$,
respectively.
Moreover,
\[
 \nabla f(\s^*)
 =
 \frac{\beta}{\sqrt2}
 \left(
   \frac5\kappa\u_N,\,
   \gamma\u_N,\,
   \frac1\kappa\rr
 \right).
\]
Since $N
 \le
 (\kappa-5)/16$,
we have
\[
 \frac{\gamma}{\sqrt N}
 =
 \frac{2\sqrt{\kappa-5}}{\kappa\sqrt N}
 \ge
 \frac8\kappa.
\]
For every non-target vertex $\s_j$,
\[
 \s_j-\s^*
 =
 \left(
   \frac{\v_j-\u_N}{\sqrt2},
   \frac{\q_j}{\sqrt2},
   -\frac{\rr}{\sqrt2}
 \right),
\]
and therefore
\[
 \begin{aligned}
 \ip{\nabla f(\s^*)}{\s_j-\s^*}
 &=
 \frac{\beta}{2}\left(
   \frac5\kappa\ip{\u_N}{\v_j-\u_N}
   +\gamma\ip{\u_N}{\q_j}
   -\frac1\kappa\norm{\rr}^2
 \right)\\
 &=
 \frac{\beta}{2}\left(
   \frac5\kappa(\rho_j-1)
   +\frac{\gamma}{\sqrt N}
   -\frac1\kappa
 \right)\\
 &\ge
 \frac{\beta}{2}\left(
   -\frac5\kappa+\frac8\kappa-\frac1\kappa
 \right)
 >
 0,
 \end{aligned}
\]
where we used
$\norm{\u_N}=\norm{\rr}=1$,
$\ip{\u_N}{\v_j}=\rho_j$,
$\ip{\u_N}{\q_j}=1/\sqrt N$,
$\rho_j\ge0$, and
$\gamma/\sqrt N\ge8/\kappa$.

Thus $\w^*=\s^*$ is indeed the unique minimizer.
Theorem~\ref{thm:deterministic-geometric-hiding} and the strong convexity of $f$ then give that whenever the returned point $\widehat\w$ is feasible,
\[
 f(\widehat\w)-f(\w^*)
 \ge
 \frac{\alpha}{2}
 \norm{\widehat\w-\w^*}^2
 \ge
 \frac{\alpha}{4},
\]
proving the theorem.

\end{proof}

\section{Randomized Lower Bounds}\label{sec:randomized-lower-bounds}
The hard instances in the lower bounds in the previous section were tailored for a specific deterministic method via a resisting oracle construction. This however leaves open the possibility that a randomized method might escape these lower bounds. Here we partially close this gap. We provide lower bounds against any randomized method. As before we shall consider the case that $\R$ is an indicator for a convex and compact set and that $f$ is strongly convex. However, while our deterministic lower bounds allowed access to the feasible set through a $k$-fold sparse projection oracle, here we shall only consider the case $k=1$.

We use a modified construction in which one target
vertex is accompanied by many non-target vertices of the same norm.  Hence a
$1$-fold sparse projection is equivalent to linear maximization over the
vertices, and typically returns a non-target vertex. Here, the set is constructed based on a randomly chosen vector $\sigma$, sampled before the interaction begins.

A randomized method is allowed at most $T$ sparse-projection calls for every
realization of its internal randomness.  All expectations below are over this
randomness once the instance is fixed.  Infeasible outputs have infinite composite error.  As in
previous section, only sparse-projection calls are counted: between them the method
may make arbitrary first-order calls to the objective.  In the objective
constructions below, the rule producing those first-order answers is
independent of the random vector $\sigma$, so these free calls do not reveal the random
instance.

\subsection{The randomized construction}

Let $\sigma\sim\operatorname{Unif}\{-1,+1\}^N$ and
$\u_\sigma=N^{-1/2}\sum_i\sigma_i\e_i$.  Write points in
$\reals^{2N}$ as $(\x,\z)$ with $\x,\z\in\reals^N$.  For $R>0$ define
\[
 \begin{aligned}
 \s_{\sigma}^*&:=(R\u_\sigma,0),\qquad
 \mD_i:=\left\{\left(\frac{R}{\sqrt2}\v,
                  \frac{R}{\sqrt2}\sigma_i\e_i\right)|
                  \norm\v=1\right\},\quad i\in[N],\\
 \mV_\sigma&:=\{\s_{\sigma}^*\}\cup\bigcup_{i=1}^N\mD_i,
 \qquad
 \mK_\sigma:=\conv(\mV_\sigma).
 \end{aligned}
\]
Note the dependence of $\mD_i$ on $\sigma$ is only through the sign $\sigma_i$.
Every point of $\mV_\sigma$ has norm $R$, and hence
$\operatorname{vert}(\mK_\sigma)=\mV_\sigma$ and $\operatorname{diam}(\mK_\sigma)\le2R$.
The point $\s_{\sigma}^*$ is the target vertex, every point in
$\bigcup_i\mD_i$ is a non-target vertex.  For $k=1$, the sparse-projection oracle projects onto
$\mC_1=\mV_\sigma$.  Since all points in $\mC_1$ have norm $R$, this is
simply linear maximization.  For an input
$(\p,\z)\in\reals^{2N}$, the target score and the largest
non-target score, divided by $R$, are
\begin{align}\label{eq:ranomLB:score}
&T_\sigma(\p,\z):=\frac{\langle(\p,\z),\s_{\sigma}^*\rangle}{R}
=\langle \p,\u_\sigma\rangle, \nonumber\\
&D_\sigma(\p,\z):=\frac{1}{R}\max_{\s\in\cup_i\mD_i}\langle(\p,\z),\s\rangle
=\frac{1}{\sqrt2}\left(\Vert\p\Vert+\max_i\sigma_i z_i\right).
\end{align}
Fix the following deterministic tie-breaking rule.  Return the target only
when $T_\sigma>D_\sigma$; otherwise return a maximum-score non-target
vertex, breaking ties by the smallest maximizing index $i$.  If $\p\ne0$,
the first block of this vertex is $R\p/(\sqrt2\norm\p)$; if $\p=0$, use
one fixed unit vector, scaled by $R/\sqrt2$, for the first block.  Thus, for
a fixed input, a non-target answer is determined by the pair $(i,\sigma_i)$
and has at most $2N$ possible values that depend on the random vector $\sigma$.

We use the following elementary concentration lemma.

\begin{lemma}%
\label{lem:simple-randomized-concentration}
For every fixed sparse-projection input $(\p,\z)$,
\[
 \Pr_\sigma\{\text{projection oracle returns }\s_{\sigma}^*\}
 \le
 \eta_N
 :=
 e^{-N/4}+2^{-N}.
\]
Moreover, for every fixed $\x\in\reals^N$,
\[
 \Pr_\sigma\!\left(
  \norm{\x-R\u_\sigma}<\frac{R}{\sqrt2}
 \right)
 \le
 e^{-N/4}
 \le
 \eta_N.
\]
\end{lemma}

\begin{proof}
For every fixed unit vector $\rr\in\reals^N$, Hoeffding's inequality gives
\begin{align}\label{eq:randomized-fixed-correlation}
 \Pr_\sigma\!\left(
  \ip\rr{\u_\sigma}>\frac1{\sqrt2}
 \right)
 \le
 e^{-N/4}.
\end{align}

Fix first an input $(\p,\z)$ and let
$M:=\max_i\sigma_i z_i$.  If $M\ge0$, then
$D_\sigma\ge\norm\p/\sqrt2$, so a target answer implies $\p\ne0$ and
$\ip{\p/\norm\p}{\u_\sigma}>1/\sqrt2$.
By~\eqref{eq:randomized-fixed-correlation}, this event has probability at
most $e^{-N/4}$.  If $M<0$, then every $z_i$ is nonzero and
$\sigma_i=-\operatorname{sign}(z_i)$ for all $i$, which is one prescribed
sign vector and has probability $2^{-N}$.  This proves the first claim.

Now fix $\x$.  If $\x=0$, the second event is impossible.  Otherwise write
$\x=t\rr$, where $t:=\norm\x>0$ and $\norm\rr=1$.  If
$\norm{\x-R\u_\sigma}<R/\sqrt2$, then
\[
 t^2+R^2-2Rt\ip\rr{\u_\sigma}
 <
 \frac{R^2}{2}.
\]
Hence
\[
 2Rt\ip\rr{\u_\sigma}
 >
 t^2+\frac{R^2}{2}
 \ge
 \sqrt2 Rt,
\]
and therefore
$\ip\rr{\u_\sigma}>1/\sqrt2$.
Thus the second claim follows again from~\eqref{eq:randomized-fixed-correlation}.
\end{proof}

\begin{theorem}[Randomized geometric lower bound]
\label{thm:simple-randomized-geometric-hiding}
For any $N\ge8$, let
$T\le\left\lfloor N/(128\log(2N))\right\rfloor$.
Every randomized method using at most $T$ calls to a $1$-fold sparse
projection oracle, while making arbitrary first-order calls whose
returned answers are generated by a rule independent of $\sigma$, has some
fixed $\sigma\in\{-1,+1\}^N$ for which its output
$\widehat\w=(\widehat\x,\widehat\z)$ satisfies
\[
 \E\norm{\widehat\x-R\u_\sigma}^2
 \ge
 \frac{R^2}{4},
\]
where the expectation is taken over the internal randomness of the method.
\end{theorem}

\begin{proof}
Fix in advance the entire sequence of random bits that the method may use,
independently of $\sigma$.  The method is then deterministic and still makes
at most $T$ sparse-projection calls.

Consider the rooted interaction tree obtained by following the method as long as every sparse-projection call returns a non-target vertex. Since the first-order answer rule is independent of \(\sigma\), each history of non-target answers, together with the fixed random bits, determines the next sparse-projection input. At depth \(d\), each node represents a possible history of \(d\) such answers and is either terminal, if the method stops, or has at most \(2N\) children, corresponding to the possible non-target answers to its next query (as we observed above). Thus, the total number of nodes through depth \(T\), including terminal histories, is at most

$$
\sum_{d=0}^{T}(2N)^d \le 2(2N)^T .
$$

Associate a bad event with every node of this tree.  At a sparse-projection
node, the bad event is that the target would be returned at its fixed input.
At a terminal node, whose first output block $\x$ is fixed, the bad event is
\[
 \norm{\x-R\u_\sigma}<\frac{R}{\sqrt2}.
\]
By Lemma~\ref{lem:simple-randomized-concentration}, every such event has
probability at most $\eta_N$.  A union bound over the tree therefore gives
\begin{align}\label{eq:randomized-tree-bad-event}
 \Pr_\sigma\{\text{some bad event occurs}\}
 \le
 2(2N)^T\eta_N.
\end{align}
By the assumed bound on $T$,
$(2N)^T\le e^{N/128}$, and hence
\[
 2(2N)^T\eta_N
 \le
 2\left(
   e^{-31N/128}
   +e^{-(\log 2-1/128)N}
 \right)
 \le
 \frac12,
\]
where the last inequality holds for every $N\ge8$.

If no bad event occurs, the actual interaction follows the non-target tree
to a terminal history and its output satisfies
\[
 \norm{\widehat\x-R\u_\sigma}
 \ge
 \frac{R}{\sqrt2}.
\]
Consequently, for every fixed realization of the method's random bits,
\[
 \E_\sigma\norm{\widehat\x-R\u_\sigma}^2
 \ge
 \frac{R^2}{2}
 \Pr_\sigma\!\left(
  \norm{\widehat\x-R\u_\sigma}\ge\frac{R}{\sqrt2}
 \right)
 \ge
 \frac{R^2}{4}.
\]
Averaging over the method's random bits and interchanging the two
expectations, we have that some fixed $\sigma$ indeed satisfies the same bound in expectation
over the method's randomness. 
\end{proof}

\subsection{The randomized lower bounds}

\begin{theorem}[The nonsmooth case]
\label{thm:simple-randomized-nonsmooth-lb}
Fix $\alpha,G>0$ and an integer $N\ge8$.  For every randomized method using
at most $\lfloor N/(128\log(2N))\rfloor$ calls to a $1$-fold sparse-projection
oracle, there exists an instance in dimension $n=2N$ with
compact feasible set, an $\alpha$-strongly convex objective whose feasible
subgradients have norm at most $G$, and a unique vertex minimizer $\w^*$%
, such that 
\[
 \E[F(\widehat\w)-F(\w^*)]
 \ge\frac{G^2}{128\alpha N},
\]
where the expectation is over the method's internal randomness.
Arbitrarily many first-order calls between sparse projections are allowed.

Consequently, there is a universal constant $c_0>0$ such that, whenever
$G^2/(\alpha\epsilon)$ is sufficiently large, every randomized method that
guarantees expected error at most $\epsilon$ on all $s=k=1$ instances
requires at least
$c_0\frac{G^2}{
 \alpha\epsilon\,\log(2G^2/(\alpha\epsilon))}$
sparse projections in the worst case.
\end{theorem}

\begin{proof}
Fix a randomized method.  Use the preceding construction with
$R=\frac{G}{4\alpha\sqrt N}$.
Consider the objective function
\[
 f(\x,\z)
 :=
 \frac{\alpha}{2}\bigl(\norm\x^2+\norm\z^2\bigr)
 +\frac{3G}{4}\norm\z,
 \qquad
 \x,\z\in\reals^N,
\]
and note it is indeed $\alpha$-strongly convex.  Moreover, every
$\g\in\partial f(\x,\z)$ can be written as
\[
 \g
 =
 \alpha(\x,\z)
 +\left(0,\frac{3G}{4}\v\right)
 \qquad\text{for some }\norm\v\le1.
\]
Since every point of the feasible set $\mK_\sigma$ has norm at most $R$, every feasible
subgradient therefore satisfies
\[
 \norm\g
 \le
 \alpha R+\frac{3G}{4}
 =
 \frac{G}{4\sqrt N}+\frac{3G}{4}
 \le G.
\]

To specify the first-order oracle, let it return
$\left(
  \alpha\x,
  \alpha\z+\frac{3G}{4}\frac{\z}{\norm\z}
 \right)
$
when $\z\ne0$, and $(\alpha\x,0)$ when $\z=0$.  This is a valid
subgradient rule that depends only on the query $(\x,\z)$ and not on
$\sigma$.

We next verify that $\w^*=\s_{\sigma}^*=(R\u_\sigma,0)$
is the unique minimizer over $\mK_\sigma$.  To certify its optimality,
consider the corresponding subgradient
$\g^*
 :=
 \left(
  \alpha R\u_\sigma,\frac{3G}{4}\u_\sigma
 \right)
 \in\partial f(\w^*)$.
For every non-target vertex
$\s=
 \left(
  \frac{R}{\sqrt2}\v,
  \frac{R}{\sqrt2}\sigma_i\e_i
 \right)$
we have,
\[
 \begin{aligned}
 \ip{\g^*}{\s-\w^*}
 &=
 \alpha R^2
 \left(
  \frac1{\sqrt2}\ip{\u_\sigma}{\v}-1
 \right)
 +\frac{3GR}{4\sqrt{2N}}\\
 &\ge
 -\alpha R^2\left(1+\frac1{\sqrt2}\right)
 +\frac{3GR}{4\sqrt{2N}}\\
 &=
 \frac{(\sqrt2-1)G^2}{16\alpha N}
 >0,
 \end{aligned}
\]
and so $\w^*=\s_{\sigma}^*$ is indeed the unique minimizer.

Thus, for every feasible output $\widehat\w=(\widehat\x,\widehat\z)$, strong
convexity implies that
\[
 \begin{aligned}
 F(\widehat\w)-F(\w^*)=
 f(\widehat\w)-f(\w^*)\ge
 \frac{\alpha}{2}\norm{\widehat\w-\w^*}^2\ge
 \frac{\alpha}{2}\norm{\widehat\x-R\u_\sigma}^2.
 \end{aligned}
\]
The same inequality holds trivially for an infeasible output because its
composite error is infinite.  

Taking expectations and using Theorem~\ref{thm:simple-randomized-geometric-hiding} indeed gives
\[
 \E\bigl[F(\widehat\w)-F(\w^*)\bigr]
 \ge
 \frac{\alpha}{2}
 \E\norm{\widehat\x-R\u_\sigma}^2
 \ge
 \frac{\alpha R^2}{8}
 =
 \frac{G^2}{128\alpha N},
\]
which proves the theorem.
\end{proof}

\begin{remark}
As in the deterministic construction of Theorem~\ref{thm:nonsmooth-sparse-lb},
the bounded-subgradient property here holds on the entire Euclidean ball
containing the feasible set. Indeed,
$\mK_\sigma\subseteq\{\w:\norm{\w}\le R\}$ and, for every
$\norm{\w}\le R$ and every $\g\in\partial f(\w)$,
\[
 \norm{\g}
 \le \alpha R+\frac{3G}{4}
 =\frac{G}{4\sqrt N}+\frac{3G}{4}
 \le G.
\]
Thus the randomized lower bound also holds under the same type of
neighborhood-access assumption discussed after
Theorem~\ref{thm:nonsmooth-sparse-lb}, and used by MOPES~\cite{MOPES}.
\end{remark}

\begin{theorem}[The smooth case]
\label{thm:simple-randomized-smooth-lb}
Fix $\alpha,\beta>0$, let $\kappa:=\beta/\alpha$,
$N:=\left\lfloor(\kappa-2)/32\right\rfloor$,
and suppose that $\kappa\ge258$.  For every randomized method using at most
$\lfloor N/(128\log(2N))\rfloor$ calls to a $1$-fold sparse-projection
oracle, there exists an instance in dimension $n=2N$ with an
$\alpha$-strongly convex and $\beta$-smooth quadratic objective, a unique
vertex minimizer $\w^*$, and feasible-set
diameter $D \leq 2$, %
such that
\[
 \E[F(\widehat\w)-F(\w^*)]
 \ge\frac{\alpha}{8}.
\]
Consequently, for $\epsilon=\alpha/16$, every randomized method that
guarantees expected error at most $\epsilon$ on all such instances requires
$\Omega\left(
  (\beta/\alpha)/\log(2\beta/\alpha)
 \right)$
calls to the $1$-fold sparse-projection oracle in the worst case.
\end{theorem}

\begin{proof}
Fix a randomized method.  Since $\kappa\ge258$, we have $N\ge8$.
Use the preceding construction with $R=1$, set
$\w^*=\s_{\sigma}^*=(\u_\sigma,0)$, and define $\gamma:=\sqrt{\kappa-2}/\kappa$.
Consider the quadratic objective
\[
 f(\x,\z)
 :=
 \frac{\beta}{2}\left(
  \frac{2}{\kappa}\norm\x^2
  +2\gamma\ip\x\z
  +\left(1-\frac1\kappa\right)\norm\z^2
 \right),
 \qquad
 \x,\z\in\reals^N.
\]
Its gradient is
\[
 \nabla f(\x,\z)
 =
 \beta\left(
  \frac2\kappa\x+\gamma\z,\,
  \gamma\x+\left(1-\frac1\kappa\right)\z
 \right).
\]
In particular, the objective and its gradient are independent of $\sigma$,
as required by Theorem~\ref{thm:simple-randomized-geometric-hiding}.

The Hessian with respect to the two blocks $(\x,\z)$ is
\[
 \nabla^2 f
 =
 \beta
 \begin{pmatrix}
  \frac2\kappa\Id_N & \gamma\Id_N\\
  \gamma\Id_N & \left(1-\frac1\kappa\right)\Id_N
 \end{pmatrix}.
\]
The trace and determinant of the $2\times2$ block, without the factor
$\beta$, are respectively
\[
 1+\frac1\kappa
 \qquad\text{and}\qquad
 \frac2\kappa\left(1-\frac1\kappa\right)-\gamma^2
 =
 \frac1\kappa.
\]
Its two distinct eigenvalues are therefore $1$ and $1/\kappa$.  Thus the exact
strong-convexity and smoothness parameters are indeed 
$\beta/\kappa=\alpha$ and $\beta$,
respectively.

We next verify that $\w^*=\s_{\sigma}^*$ is the unique minimizer over the feasible set $\mK_\sigma$.  At
$\w^*$,
\[
 \nabla f(\w^*)
 =
 \beta\left(
  \frac2\kappa\u_\sigma,\,
  \gamma\u_\sigma
 \right).
\]
Since $N\le(\kappa-2)/32$,
we have
\[
 \frac{\gamma}{\sqrt N}
 =
 \frac{\sqrt{\kappa-2}}{\kappa\sqrt N}
 \ge\frac{4\sqrt2}{\kappa}.
\]
For every non-target vertex
\[
 \s
 =
 \left(
  \frac1{\sqrt2}\v,
  \frac1{\sqrt2}\sigma_i\e_i
 \right),
 \qquad \norm\v=1,
\]
using $\ip{\u_\sigma}{\sigma_i\e_i}=1/\sqrt N$, we have
\[
 \begin{aligned}
 \ip{\nabla f(\w^*)}{\s-\w^*}
 &=
 \beta\left[
  \frac2\kappa
  \left(
   \frac1{\sqrt2}\ip{\u_\sigma}{\v}-1
  \right)
  +\frac{\gamma}{\sqrt{2N}}
 \right]\\
 &\ge
 \frac{\beta}{\kappa}
 \left[
  -2\left(1+\frac1{\sqrt2}\right)+4
 \right] >0.
 \end{aligned}
\]
Thus the constrained first-order optimality condition holds at $\w^*$, and
strong convexity makes it the unique minimizer.

For every feasible output
$\widehat\w=(\widehat\x,\widehat\z)$, strong convexity and the optimality
condition give
\[
 \begin{aligned}
 F(\widehat\w)-F(\w^*)
 &=
 f(\widehat\w)-f(\w^*)\ge
 \frac{\alpha}{2}\norm{\widehat\w-\w^*}^2
 \ge
 \frac{\alpha}{2}\norm{\widehat\x-\u_\sigma}^2.
 \end{aligned}
\]
The same inequality holds trivially for an infeasible output because its
composite error is infinite.  Taking expectations and applying
Theorem~\ref{thm:simple-randomized-geometric-hiding}, we obtain, for some
fixed $\sigma$,
\[
 \E\bigl[F(\widehat\w)-F(\w^*)\bigr]
 \ge
 \frac{\alpha}{2}
 \E\norm{\widehat\x-\u_\sigma}^2
 \ge
 \frac{\alpha}{8},
\]
which proves the theorem.

\end{proof}

\section{WPO-based Upper Bounds}\label{sec:qg-upper-bounds}
In this section we turn to consider upper bounds for Problem \eqref{eq:optProb}. As opposed to our lower bounds constructions that considered stronger assumptions than those described in the Introduction, here we consider the original weaker assumptions given in the Introduction. In particular $F$ corresponds to a general convex composite model, i.e.,  $\R$ need not be an indicator for a convex set, and $f$ is not assumed strongly convex and we do not assume a unique minimizer; we only assume $F$ satisfies the quadratic growth condition \eqref{eq:qg-condition}, but can have more than one minimizer. 
Throughout this section $\mA$ denotes a weak proximal oracle as defined in Definition~\ref{def:wpo}.

We present algorithms for both the smooth and nonsmooth regimes that access the function $\R$ only through the weak proximal oracle $\mA$. For the smooth case the algorithm and its linear convergence rate result are not new and have appeared in several previous works, e.g., \cite{garber2021improved, garber2019fast, garber2023faster, garber2025weak}. Given that both the description of the algorithm and the convergence analysis are very short, we include it for completeness. For the nonsmooth case, while the presented algorithm and convergence analysis are based on previously introduced techniques, to the best of our knowledge, such result was not previously known, and the technical details are far more intricate than in the smooth case.

\subsection{The smooth case}
The algorithm is given as Algorithm \ref{alg:qg-smooth-wpo}.
\begin{algorithm}[t]
\caption{Weak proximal gradient method under quadratic growth}
\label{alg:qg-smooth-wpo}
\begin{algorithmic}[1]
\Require $\x_{\rm in}\in\dom(\R)$, $\alpha,\beta>0$, number of iterations
$T\ge0$, weak proximal oracle $\mA$
\State $\x_0\gets\x_{\rm in}$
\For{$t=0,\ldots,T-1$}
  \State $\v_t\gets
  \mA\left(
    \x_t-\dfrac{2}{\alpha}\nabla f(\x_t),
    \dfrac{\alpha}{4}
  \right)$
  \State $\x_{t+1}\gets
  \left(1-\dfrac{\alpha}{2\beta}\right)\x_t
  +\dfrac{\alpha}{2\beta}\v_t$
\EndFor
\State \Return $\x_T$
\end{algorithmic}
\end{algorithm}

\begin{theorem}
\label{thm:qg-smooth-wpo-upper}
For any $\x_{\rm in}\in\dom(\R)$ and any integer $T\ge0$, let $\x_T$
denote the output of Algorithm \ref{alg:qg-smooth-wpo}. Then
\begin{align}\label{eq:qg-smooth-contraction}
 F(\x_T)-F^*
 \le
 \left(1-\frac{\alpha}{4\beta}\right)^T
 \bigl(F(\x_{\rm in})-F^*\bigr).
\end{align}
In particular, Algorithm \ref{alg:qg-smooth-wpo} uses  $T$ weak
proximal oracle calls and $T$ gradient evaluations.
\end{theorem}

\begin{proof}
Fix an iteration $t$ and choose a closest optimizer
$\x_t^*\in\argmin_{\x^*\in\mX^*}\norm{\x_t-\x^*}$. Applying the
weak proximal guarantee \eqref{eq:wpo-def} to the oracle call in Algorithm  \ref{alg:qg-smooth-wpo} with comparison
point $\x_t^*$ gives,
\[
 \frac{\alpha}{4}
 \norm{\v_t-\left(\x_t-\frac{2}{\alpha}\nabla f(\x_t)\right)}^2+\R(\v_t)
 \le
 \frac{\alpha}{4}
 \norm{\x_t^*-\left(\x_t-\frac{2}{\alpha}\nabla f(\x_t)\right)}^2
 +\R(\x_t^*).
\]
Expanding the two squared norms gives,
\begin{align}\label{eq:qg-smooth-wpo-expanded}
 &\ip{\nabla f(\x_t)}{\v_t-\x_t}
  + \frac{\alpha}{4}\Vert{\v_t-\x_t}\Vert^2 +\R(\v_t)
  \nonumber \\
 &\leq \ip{\nabla f(\x_t)}{\x_t^*-\x_t}
  + \frac{\alpha}{4}\Vert{\x_t^*-\x_t}\Vert^2 +\R(\x_t^*)
\end{align}
By $\beta$-smoothness of $f$, convexity of $\R$, and the update of
$\x_{t+1}$ in Algorithm~\ref{alg:qg-smooth-wpo},
\begin{align*}
 F(\x_{t+1})
 &\le
 f(\x_t)
 +\frac\alpha{2\beta}\ip{\nabla f(\x_t)}{\v_t-\x_t}
 +\frac{\alpha^2}{8\beta}\norm{\v_t-\x_t}^2 \\
 &~~+\left(1-\frac\alpha{2\beta}\right)\R(\x_t)
 +\frac\alpha{2\beta}\R(\v_t) \\
 &\leq 
 f(\x_t)
 +\frac\alpha{2\beta}\ip{\nabla f(\x_t)}{\x_t^*-\x_t}
 +\frac{\alpha^2}{8\beta}\norm{\x_t^*-\x_t}^2 \\
 &~~+\left(1-\frac\alpha{2\beta}\right)\R(\x_t)
 +\frac\alpha{2\beta}\R(\x_t^*),  
\end{align*}
where the last inequality is due to \eqref{eq:qg-smooth-wpo-expanded}.

Convexity of $f$ also gives
\[
 f(\x_t)-f(\x_t^*)
 +\ip{\nabla f(\x_t)}{\x_t^*-\x_t}\le0.
\]
Subtracting $F^*=F(\x_t^*)$ from both sides and using the last inequality yields
\begin{align*}
 F(\x_{t+1})-F^*
 &\le
 \left(1-\frac\alpha{2\beta}\right)
 \bigl(F(\x_t)-F^*\bigr)
 +\frac{\alpha^2}{8\beta}\norm{\x_t-\x_t^*}^2.
\end{align*}
Finally, quadratic growth gives
\[
 F(\x_{t+1})-F^*
 \le
 \left(1-\frac\alpha{4\beta}\right)
 \bigl(F(\x_t)-F^*\bigr),
\]
which proves the theorem.
\end{proof}

\begin{remark}[use of line-search in Algorithm \ref{alg:qg-smooth-wpo}]
Algorithm \ref{alg:qg-smooth-wpo} only requires the smoothness parameter $\beta$ in order to set the convex combination parameter when forming the new point $\x_{t+1}$. It is a simple observation that replacing this parameter with exact line-search on the $[0,1]$ interval, i.e., taking $\x_{t+1}\gets(1-\eta_t)\x_t+\eta_t\v_t$ with $\eta_t\gets\argmin_{\eta\in[0,1]}F((1-\eta)\x_t + \eta\v_t)$, does not change the consequences of Theorem \ref{thm:qg-smooth-wpo-upper}.
\end{remark}

\subsection{The nonsmooth case}
We now turn to consider a WPO-based upper bound for the nonsmooth case, which is significantly more intricate than the previous smooth case. Our Algorithm, given below as Algorithm \ref{alg:qg-nonsmooth-wpo}, iteratively invokes a one-stage method  described in Algorithm \ref{alg:qg-sliding-stage}. This one-stage method and its analysis are heavily based on the optimistic primal-dual and linear optimization oracle-based method presented in  \cite{AsgariNeely}. Nevertheless, there are also some important differences: \cite{AsgariNeely} considers the constrained setting only (i.e., $\R$ is an indicator for a convex and compact set), it assumes the subgradients of the objective are globally bounded, and it assumes access to a linear optimization oracle. Also they provide only a $1/\epsilon^2$-type convergence rate, which is understandable given their use of a linear optimization oracle. Here, we assume the more general proximal setting, we do not require a global bound on subgradients,  we assume access through a weak proximal oracle, and we prove a $1/\epsilon$-type convergence rate, under quadratic growth. These differences naturally require certain changes to the algorithm and its analysis. For instance, we add localization steps (the projections onto the set $\mY$ in Algorithm \ref{alg:qg-sliding-stage}) in order to rely only on a local bound on subgradients, rather than a global one. Additionally, while \cite{AsgariNeely} uses a single-loop algorithm and each iteration makes one subgradient oracle call and one linear optimization oracle call, our Algorithm \ref{alg:qg-sliding-stage} uses a double loop in order to balance correctly between the number of calls to the subgradient oracle and the weak proximal oracle.

In what follows, let $\x_{\rm in}\in\dom(\R)$ and let $\Delta_0$ be any known upper bound
satisfying $F(\x_{\rm in})-F^*\le\Delta_0$. Define the initial sublevel set $\mathcal L_0
 :=
 \left\{
  \x\in\dom(\R):F(\x)-F^*\le\Delta_0
 \right\}$.
Assume that, for some $G>0$, $\norm{\g}\le G$ for every
 $\x\in\mathcal L_0$, $\g\in\partial f(\x)$.
Set $D_0:=\sqrt{2\Delta_0/\alpha}$,
 $\mathcal U_0
 :=
 \left\{
  \u\in\mathbb V | 
  \dist(\u,\mathcal L_0)\le D_0
 \right\}$,
and define
$$\overline G_0
 :=
 \max\left\{
  G,\,
  \sup\left\{
   \norm{\g}|
   \u\in\mathcal U_0,
   \g\in\partial f(\u)
  \right\}
 \right\}.$$

\begin{algorithm}[t]
\caption{One-stage nonsmooth weak-proximal primal-dual method}
\label{alg:qg-sliding-stage}
\begin{algorithmic}[1]
\Require center $\c\in\dom(\R)$, radius $D>0$, $G,\overline G_0>0$, integers $T,K\ge1$, weak proximal oracle $\mA$
\State $\eta\gets G/(D\sqrt T)$, $\rho\gets(\overline G_0/D)\sqrt{T/K}$
\State $\mY\gets\{\y:\norm{\y-\c}\le D\}$
\State $\q_1\gets0$, $\y_{1,1}\gets\c$
\For{$t=1,\ldots,T$}
  \State $\v_t\gets\mA\left(\c-\dfrac{\q_t}{2\eta},\eta\right)$
  \For{$s=1,\ldots,K$}
    \State choose $\g_{t,s}\in\partial f(\y_{t,s})$
    \State $\displaystyle
      \y_{t,s+1}\gets
      \Pi_{\mY}\left(
       \frac{\rho K\y_{t,s}+\q_t+\eta\v_t-\g_{t,s}}
       {\rho K+\eta}
      \right)$
  \EndFor
  \State $\displaystyle
    \widehat{\y}_t\gets\frac1K\sum_{s=1}^K\y_{t,s+1}$,
    $\quad
    \q_{t+1}\gets\q_t+\eta(\v_t-\widehat{\y}_t)$,
    $\quad
    \y_{t+1,1}\gets\y_{t,K+1}$
\EndFor
\State \Return $\overline\v_T\gets T^{-1}\sum_{t=1}^T\v_t$
\end{algorithmic}
\end{algorithm}

\begin{lemma}
\label{lem:qg-nonsmooth-sliding-stage}
Let $\c\in\mathcal L_0$, let $D>0$, and suppose that there is
$\x_\c^*\in\mX^*$ such that $\norm{\c-\x_\c^*}\le D$. Set
$\mY:=\{\y:\norm{\y-\c}\le D\}$,
and suppose that $\norm{\g}\le\overline G_0$ for every $\y\in\mY$ and
every $\g\in\partial f(\y)$. Fix positive integers $T,K$ satisfying
$K\ge \overline G_0^2/G^2$,
and run Algorithm~\ref{alg:qg-sliding-stage} with center $\c$, radius $D$,
constants $G,\overline G_0$, and integers $T,K$. If
$6GD/\sqrt T<\Delta_0$,
then its output
$\overline\v_T:=T^{-1}\sum_{t=1}^T\v_t$ belongs to $\mathcal L_0$ and
\begin{align}\label{eq:qg-one-stage-simplified-bound}
 F(\overline\v_T)-F^*
 \le\frac{6GD}{\sqrt T}.
\end{align}
The algorithm uses $T$ weak proximal oracle calls and $TK$ subgradient
calls to $f$.
\end{lemma}

\begin{proof}
Denote $\x^*:=\x_\c^*\in\mY$. Expanding the
WPO guarantee for the call in the algorithm and using $\x^*$ as the
comparison point (similarly to the derivation of Eq. \eqref{eq:qg-smooth-wpo-expanded} in the smooth case) gives
\begin{align}\label{eq:qg-wpo-dual-comparison}
 \ip{\q_t}{\v_t-\x^*}
 +\R(\v_t)-\R(\x^*)
 \le
 \eta\left(
  \norm{\x^*-\c}^2-\norm{\v_t-\c}^2
 \right).
\end{align}

For fixed $t$ and $s$, define
\begin{align*}
 \Phi_{t,s}(\y)
 &:=
 \ip{\g_{t,s}-\q_t}{\y}
 +\frac\eta2\norm{\y-\v_t}^2
 +\frac{\rho K}{2}\norm{\y-\y_{t,s}}^2.
\end{align*}
Completing the square gives
\begin{align*}
 \Phi_{t,s}(\y)
 &=
 \frac{\rho K+\eta}{2}
 \norm{
  \y-
  \frac{\rho K\y_{t,s}+\q_t+\eta\v_t-\g_{t,s}}
       {\rho K+\eta}
 }^2
 +\mathrm{const}.
\end{align*}
Consequently, the projected update in
Algorithm~\ref{alg:qg-sliding-stage} is exactly
\begin{align*}
 \y_{t,s+1}
 =
 \argmin_{\y\in\mY}\Phi_{t,s}(\y).
\end{align*}

The function $\Phi_{t,s}$ is $(\rho K+\eta)$-strongly convex. Since
$\x^*\in\mY$, this implies that
\begin{align*}
 \Phi_{t,s}(\y_{t,s+1})
 +\frac{\rho K+\eta}{2}
  \norm{\x^*-\y_{t,s+1}}^2
 \le
 \Phi_{t,s}(\x^*).
\end{align*}

Expanding and dividing through by $K$ gives
\begin{align*}
 &\frac1K
  \ip{\g_{t,s}-\q_t}{\y_{t,s+1}-\x^*}
 +\frac{\eta}{2K}\norm{\y_{t,s+1}-\v_t}^2
 +\frac\rho2\norm{\y_{t,s+1}-\y_{t,s}}^2\\
 &\qquad\le
 \frac{\eta}{2K}\norm{\x^*-\v_t}^2
 +\frac\rho2\left(
  \norm{\x^*-\y_{t,s}}^2
  -\norm{\x^*-\y_{t,s+1}}^2
 \right)
 -\frac{\eta}{2K}\norm{\x^*-\y_{t,s+1}}^2.
\end{align*}
Using
\begin{align*}
 \ip{\g_{t,s}-\q_t}{\y_{t,s+1}-\x^*}
 &=
 \ip{\g_{t,s}}{\y_{t,s}-\x^*}
 +\ip{\g_{t,s}}{\y_{t,s+1}-\y_{t,s}}
 +\ip{\q_t}{\x^*-\y_{t,s+1}},
\end{align*}
we obtain
\begin{align*}
 &\frac1K\ip{\g_{t,s}}{\y_{t,s}-\x^*}
 +\frac1K\ip{\q_t}{\x^*-\y_{t,s+1}}
 +\frac{\eta}{2K}\norm{\y_{t,s+1}-\v_t}^2\\
 &\quad
 +\frac1K\ip{\g_{t,s}}{\y_{t,s+1}-\y_{t,s}}
 +\frac\rho2\norm{\y_{t,s+1}-\y_{t,s}}^2\\
 &\qquad\le
 \frac{\eta}{2K}\norm{\x^*-\v_t}^2
 +\frac\rho2\left(
  \norm{\x^*-\y_{t,s}}^2
  -\norm{\x^*-\y_{t,s+1}}^2
 \right).
\end{align*}
Young's inequality and
$\norm{\g_{t,s}}\le\overline G_0$, give
\begin{align*}
 &\frac1K\ip{\g_{t,s}}{\y_{t,s+1}-\y_{t,s}}
 +\frac\rho2\norm{\y_{t,s+1}-\y_{t,s}}^2\ge
 -\frac{\norm{\g_{t,s}}^2}{2\rho K^2}
 \ge
 -\frac{\overline G_0^2}{2\rho K^2}.
\end{align*}
Plugging this inequality into the previous one and using the convexity of $f$, gives
\begin{align}
 &\frac{f(\y_{t,s})-f(\x^*)}{K}
 +\frac1K\ip{\q_t}{\x^*-\y_{t,s+1}}
 +\frac\eta{2K}\norm{\y_{t,s+1}-\v_t}^2
 \notag\\
 &\qquad\le
 \frac\eta{2K}\norm{\x^*-\v_t}^2
 +\frac\rho2\left(
  \norm{\x^*-\y_{t,s}}^2
  -\norm{\x^*-\y_{t,s+1}}^2
 \right)
 +\frac{\overline G_0^2}{2\rho K^2}.
 \label{eq:qg-inner-prediction-bound}
\end{align}
Jensen's inequality and the definition of $\widehat{\y}_t$ give
\[
 \frac1K\sum_{s=1}^K
 \norm{\y_{t,s+1}-\v_t}^2
 \ge
 \norm{\widehat{\y}_t-\v_t}^2.
\]
Additionally, the dual update gives
\begin{align*}
 \frac{1}{2\eta}
 \left(\norm{\q_{t+1}}^2-\norm{\q_t}^2\right)
 &=
 \frac{1}{2\eta}
 \left(
  \norm{\q_t+\eta(\v_t-\widehat{\y}_t)}^2
  -\norm{\q_t}^2
 \right)\\
 &=
 \ip{\q_t}{\v_t-\widehat{\y}_t}
 +\frac{\eta}{2}\norm{\v_t-\widehat{\y}_t}^2.
\end{align*}
Using $\ip{\q_t}{\x^*-\widehat{\y}_t}
 =
 \ip{\q_t}{\x^*-\v_t}
 +\ip{\q_t}{\v_t-\widehat{\y}_t}$,
we obtain the identity
\begin{align*}
 &\ip{\q_t}{\x^*-\widehat{\y}_t}
 +\frac{\eta}{2}\norm{\widehat{\y}_t-\v_t}^2
 =
 \ip{\q_t}{\x^*-\v_t}
 +\frac{1}{2\eta}\left(
  \norm{\q_{t+1}}^2-\norm{\q_t}^2
 \right).
\end{align*}
Consequently, summing~\eqref{eq:qg-inner-prediction-bound} over $s$ and using the last two observations, gives
\begin{align}
 &\frac1K\sum_{s=1}^K
  \bigl(f(\y_{t,s})-f(\x^*)\bigr)
 +\frac1{2\eta}\left(
  \norm{\q_{t+1}}^2-\norm{\q_t}^2
 \right)
 \notag\\
 &\le
 \ip{\q_t}{\v_t-\x^*}
 +\frac\eta2\norm{\x^*-\v_t}^2
 +\frac{\overline G_0^2}{2\rho K}
 +\frac\rho2\left(
  \norm{\x^*-\y_{t,1}}^2
  -\norm{\x^*-\y_{t,K+1}}^2
 \right).
 \label{eq:qg-block-prediction-bound}
\end{align}
Adding $\R(\v_t)-\R(\x^*)$ to both sides and applying \eqref{eq:qg-wpo-dual-comparison}, we get
\begin{align*}
 &\frac1K\sum_{s=1}^K
  \bigl(f(\y_{t,s})-f(\x^*)\bigr)
 +\R(\v_t)-\R(\x^*)
 +\frac1{2\eta}\left(
  \norm{\q_{t+1}}^2-\norm{\q_t}^2
 \right)\\
 &\qquad\le
 \eta\left(
  \norm{\x^*-\c}^2-\norm{\v_t-\c}^2
  +\frac12\norm{\x^*-\v_t}^2
 \right)
 +\frac{\overline G_0^2}{2\rho K}\\
 &\qquad\quad
 +\frac\rho2\left(
  \norm{\x^*-\y_{t,1}}^2
  -\norm{\x^*-\y_{t,K+1}}^2
 \right).
\end{align*}
Note,
\begin{align*}
 &\norm{\x^*-\c}^2-\norm{\v_t-\c}^2
 +\frac12\norm{\x^*-\v_t}^2\\
 &~~=
 2\norm{\x^*-\c}^2
 -\frac12\norm{(\x^*-\c)+(\v_t-\c)}^2\le 2D^2.
\end{align*}
Thus, the entire outer iteration yields
\begin{align}
 &\frac1K\sum_{s=1}^K
  \bigl(f(\y_{t,s})-f(\x^*)\bigr)
 +\R(\v_t)-\R(\x^*)
 +\frac1{2\eta}\left(
  \norm{\q_{t+1}}^2-\norm{\q_t}^2
 \right)
 \notag\\
 &\qquad\le
 \frac{\overline G_0^2}{2\rho K}
 +2\eta D^2
 +\frac\rho2\left(
  \norm{\x^*-\y_{t,1}}^2
  -\norm{\x^*-\y_{t,K+1}}^2
 \right).
 \label{eq:qg-block-wpo-bound}
\end{align}

Summing~\eqref{eq:qg-block-wpo-bound} over $t$, dividing by $T$, using
$\q_1=0$, $\y_{1,1}=\c$,
$\y_{t+1,1}=\y_{t,K+1}$, and slightly rearranging, we get
\begin{align}
 &\frac1{TK}
  \sum_{t=1}^T\sum_{s=1}^K f(\y_{t,s})
 +\frac1T\sum_{t=1}^T\R(\v_t)-F^*
 \notag\\
 &\qquad\le
 \frac{\overline G_0^2}{2\rho K}
 +2\eta D^2 +\frac{\rho D^2}{2T} -\frac1{2\eta T}\norm{\q_{T+1}}^2.
 \label{eq:qg-stage-potential-bound}
\end{align}

It remains to transfer this estimate from the auxiliary average to the WPO
average. Define
\[
 \overline\y^-
 :=
 \frac1{TK}\sum_{t=1}^T\sum_{s=1}^K\y_{t,s},
 \qquad
 \overline\y^+
 :=
 \frac1{TK}\sum_{t=1}^T\sum_{s=1}^K\y_{t,s+1}.
\]
Averaging the differences of dual updates and using the continuity of the auxiliary sequence
between successive outer iterations gives
\[
 \q_{T+1}
 =
 \eta T(\overline\v_T-\overline\y^+),
 \qquad
 \overline\y^+-\overline\y^-
 =
 \frac{\y_{T,K+1}-\y_{1,1}}{TK}.
\]
Since $\y_{T,K+1}, \y_{1,1}$ are in $\mY$, this yields
\begin{align}\label{eq:qg-average-coupling}
 \norm{\overline\v_T-\overline\y^-}
 \le
 \frac{\norm{\q_{T+1}}}{\eta T}
 +\frac{2D}{TK}.
\end{align}

For $0\le\theta\le1$, set
\[
 \v_\theta
 :=
 (1-\theta)\x^*+\theta\overline\v_T,
 \qquad
 \y_\theta
 :=
 (1-\theta)\x^*+\theta\overline\y^-.
\]
Whenever $\v_\theta\in\mathcal L_0$, choose
$\g_\theta\in\partial f(\v_\theta)$. The bound
$\norm{\g_\theta}\le G$, convexity, and
\eqref{eq:qg-average-coupling} imply
\begin{align*}
 f(\v_\theta)-f(\y_\theta)
 &\le
 \ip{\g_\theta}{\v_\theta-\y_\theta}\le
 \theta\left(
  \frac{G}{\eta T}\norm{\q_{T+1}}
  +\frac{2GD}{TK}
 \right).
\end{align*}
Also, convexity gives
\begin{align*}
 f(\y_\theta)
 &\le
 (1-\theta)f(\x^*)
 +\frac{\theta}{TK}
  \sum_{t=1}^T\sum_{s=1}^K f(\y_{t,s}),\\
 \R(\v_\theta)
 &\le
 (1-\theta)\R(\x^*)
 +\frac{\theta}{T}
  \sum_{t=1}^T\R(\v_t).
\end{align*}
Combining the preceding bounds, whenever
$\v_\theta\in\mathcal L_0$, gives
\begin{align*}
 F(\v_\theta)-F^*
 &=
 f(\v_\theta)-f(\y_\theta)
 +f(\y_\theta)+\R(\v_\theta)-F^*\\
 &\le
 \theta\Bigg[
  \frac1{TK}\sum_{t=1}^T\sum_{s=1}^K f(\y_{t,s})
  +\frac1T\sum_{t=1}^T\R(\v_t)-F^*
  +\frac{G}{\eta T}\norm{\q_{T+1}}
  +\frac{2GD}{TK}
 \Bigg].
\end{align*}
Applying \eqref{eq:qg-stage-potential-bound} and using
\[
 \frac{G}{\eta T}\norm{\q_{T+1}}
 -\frac1{2\eta T}\norm{\q_{T+1}}^2
 =
 \frac{G^2}{2\eta T}
 -\frac1{2\eta T}
  \left(\norm{\q_{T+1}}-G\right)^2
 \le
 \frac{G^2}{2\eta T},
\]
we have that whenever $\v_\theta\in\mathcal L_0$,
\begin{align}
 F(\v_\theta)-F^*
 \le
 \theta\left(
  \frac{\overline G_0^2}{2\rho K}
  +\frac{\rho D^2}{2T}
  +2\eta D^2
  +\frac{G^2}{2\eta T}
  +\frac{2GD}{TK}
 \right).
 \label{eq:qg-segment-bound}
\end{align}

For the parameters in Algorithm~\ref{alg:qg-sliding-stage},
\[
 \frac{\overline G_0^2}{2\rho K}
 +\frac{\rho D^2}{2T}
 =
 \frac{\overline G_0D}{\sqrt{TK}}
 \le
 \frac{GD}{\sqrt T},
 \qquad
 2\eta D^2
 +\frac{G^2}{2\eta T}
 =
 \frac{5GD}{2\sqrt T}.
\]
Moreover, $\frac{2GD}{TK}
 \le
 \frac{2GD}{\sqrt T}$.
Thus the RHS of
\eqref{eq:qg-segment-bound} is at most
\begin{align}\label{eq:nonsmoothUB:1}
 \theta\frac{11GD}{2\sqrt T}
 <
 \theta\frac{6GD}{\sqrt T}
 <
 \theta\Delta_0.
\end{align}

Now suppose by way of contradiction that
$\overline\v_T\notin\mathcal L_0$. Since
$\x^*\in\mX^*\subseteq\mathcal L_0$ and $F$ is closed and convex, there is
a $\theta_0\in(0,1)$ such that
$F(\v_{\theta_0})-F^*=\Delta_0$.
However, applying \eqref{eq:qg-segment-bound} at $\theta_0$, together with \eqref{eq:nonsmoothUB:1}, then leads to a contradiction. Hence $\overline\v_T\in\mathcal L_0$. We may therefore
take $\theta=1$ in \eqref{eq:qg-segment-bound} which,  together with \eqref{eq:nonsmoothUB:1}, proves
\eqref{eq:qg-one-stage-simplified-bound}.
\end{proof}

Algorithm \ref{alg:qg-nonsmooth-wpo} given below  iteratively invokes the one-stage Algorithm \ref{alg:qg-sliding-stage} with a standard shrinking mechanism.

\begin{algorithm}[t]
\caption{Nonsmooth weak-proximal method under quadratic growth}
\label{alg:qg-nonsmooth-wpo}
\begin{algorithmic}[1]
\Require $\x_{\rm in}\in\dom(\R)$, $\Delta_0$, $\alpha$, target accuracy $\epsilon>0$, local subgradient bounds $G, \overline G_0$, weak proximal oracle $\mA$
\State $\x\gets\x_{\rm in}$, $\Delta\gets\Delta_0$
\State $K\gets\left\lceil \overline G_0^2/G^2\right\rceil$
\While{$\Delta>\epsilon$}
  \State $D\gets\sqrt{2\Delta/\alpha}$
  \State $T\gets\left\lceil288G^2/(\alpha\Delta)\right\rceil$
\State $\x\gets$  output of Algorithm~\ref{alg:qg-sliding-stage} with center $\x$, radius $D$, constants $G,\overline G_0$, and integers $T,K$
  \State $\Delta\gets6G\sqrt{2\Delta/(\alpha T)}$
\EndWhile
\State \Return $\x$
\end{algorithmic}
\end{algorithm}

\begin{theorem}%
\label{thm:qg-nonsmooth-wpo-upper}
Algorithm~\ref{alg:qg-nonsmooth-wpo} returns
$\widehat\x\in\dom(\R)$ with $F(\widehat\x)-F^*\le\epsilon$ using
\begin{align}\label{eq:qg-nonsmooth-wpo-complexity}
 O\left(
  1+\frac{G^2}{\alpha\epsilon}
  +\log\left(
    1+\log\left(1+\frac{\alpha\Delta_0}{G^2}\right)
  \right)
 \right)
\end{align}
weak proximal oracle calls and
\begin{align}\label{eq:qg-nonsmooth-fo-complexity}
 O\left(
  \frac{\overline G_0^2}{\alpha\epsilon}
  +\frac{\overline G_0^2}{G^2}\left({1+
   \log\left(
    1+\log\left(1+\frac{\alpha\Delta_0}{G^2}\right)
   \right)}\right)
 \right)
\end{align}
subgradient calls to $f$.  
\end{theorem}

\begin{proof}
We first establish by induction that every invocation of
Algorithm~\ref{alg:qg-sliding-stage} satisfies all the assumptions of
Lemma~\ref{lem:qg-nonsmooth-sliding-stage}. In fact, it suffices to show that at the beginning of each iteration of the while loop we have that $\x\in\mathcal{L}_0$ and that there exists $\x^*\in\mX^*$ such that $\Vert{\x-\x^*}\Vert \leq D$: the definitions of $K,T,D, \Delta$ in the algorithm imply the remaining conditions: $\mY\subseteq\mathcal{U}_0$ (which implies the upper-bound on the subgradients), $K \geq \overline{G}_0^2/G^2$, and $6GD/\sqrt{T} < \Delta_0$.  As for the condition $\Vert{\x-\x^*}\Vert \leq D$, note that if we have $F(\x) - F^* \leq \Delta$, then quadratic growth implies that indeed there is $\x^*\in\mX^*$ such that $\Vert{\x-\x^*}\Vert \leq \sqrt{2\Delta/\alpha}=D$. Thus, it suffices to prove that at the beginning of each iteration $\x\in\mathcal{L}_0$ and $F(\x) - F^* \leq \Delta$. This holds by assumption for the first iteration. Now suppose these conditions hold at the beginning of some iteration and so the assumptions of Lemma~\ref{lem:qg-nonsmooth-sliding-stage} hold, and let $\x^+, \Delta^+$ denote the updates to $\x,\Delta$ after the call to Algorithm~\ref{alg:qg-sliding-stage}. By Lemma~\ref{lem:qg-nonsmooth-sliding-stage} we have that $\x^+\in\mathcal{L}_0$, and  
\begin{align}\label{eq:qg-unified-recurrence}
F(\x^+) - F^* \leq \frac{6GD}{\sqrt{T}} = \Delta^+ \leq \frac{\Delta}{2},
\end{align}
where the equality and the last inequality follow from the definitions of $D,T,\Delta$.
 
Thus, the assumptions of  Lemma~\ref{lem:qg-nonsmooth-sliding-stage} continue to hold for the next iteration, and the induction holds.
This also proves that Algorithm~\ref{alg:qg-nonsmooth-wpo} terminates with $F(\widehat\x)-F^*\le\epsilon$.

It remains to count oracle calls.  It is convenient to distinguish between stages according to
whether $\Delta>288G^2/\alpha$.  Consider first a stage for which $\Delta>288G^2/\alpha$, then $T=1$.
Let $\Delta^+$ denote the value assigned to the maintained bound at the end
of this stage. By the update in the algorithm, $\Delta^+
 = 6G\sqrt{2\Delta/\alpha}$.
Define the corresponding normalized bounds $q:=\alpha\Delta/G^2$, $q^+:=\alpha\Delta^+/G^2$.
The recurrence  is
therefore
\[
 q^+
 =
 \frac{\alpha}{G^2}
 6G\sqrt{\frac{2\Delta}{\alpha}}
 =
 6\sqrt{\frac{2\alpha\Delta}{G^2}}
 =
 6\sqrt{2q}.
\]
Consequently, the number $R$ of stages for which
$\Delta>288G^2/\alpha$ satisfies
\begin{align}\label{eq:qg-coarse-count}
 R
 =O\left(
  1+\log\left(
    1+\log\left(1+\frac{\alpha\Delta_0}{G^2}\right)
  \right)
 \right).
\end{align}
Each such stage uses one weak proximal call and $K$ subgradient calls.

Consider now the remaining stages, for which
$\Delta\le288G^2/\alpha$.  Let their successive maintained bounds be
$\Delta_0',\Delta_1',\ldots,\Delta_{J-1}'$, where each one is larger than
$\epsilon$.  By~\eqref{eq:qg-unified-recurrence},
$\Delta_{j+1}'\le\Delta_j'/2$, and hence
\[
 \frac1{\Delta_j'}
 \le
 2^{-(J-1-j)}\frac1{\Delta_{J-1}'}.
\]
Since $\Delta_{J-1}'>\epsilon$,
\begin{align}\label{eq:qg-unified-reciprocal-sum}
 \sum_{j=0}^{J-1}\frac1{\Delta_j'}
 \le\frac{2}{\Delta_{J-1}'}
 <\frac2\epsilon.
\end{align}
For these stages, $T_j 
 \le
 1+(288G^2)/(\alpha\Delta_j')$.
Also, because the maintained bound decreases by at least a factor two (see \eqref{eq:qg-unified-recurrence}) and
starts this regime at most $288G^2/\alpha$,
\[
 J
 =
 O\left(
  1+\log\left(1+\frac{G^2}{\alpha\epsilon}\right)
 \right)
 =
 O\left(1+\frac{G^2}{\alpha\epsilon}\right).
\]
Together with~\eqref{eq:qg-unified-reciprocal-sum}, this gives
\[
 \sum_{j=0}^{J-1}T_j \leq \sum_{j=0}^{J-1}\left({1 + \frac{288G^2}{\alpha\Delta_j'}}\right)
 =
 O\left(1+\frac{G^2}{\alpha\epsilon}\right).
\]
Combining this bound with~\eqref{eq:qg-coarse-count}, the total number
 of weak proximal oracle calls indeed satisfies \eqref{eq:qg-nonsmooth-wpo-complexity}.

Since every outer iteration of Algorithm~\ref{alg:qg-sliding-stage} uses one
weak proximal oracle call and exactly $K$ subgradient calls, the
total number of subgradient calls is $K$ times that of the WPO, which proves \eqref{eq:qg-nonsmooth-fo-complexity}.
\end{proof}

\begin{remark}[The constrained case]
Suppose $\R=I_{\mK}$, where $\mK$ is convex and compact, and suppose
$\mK$ is enclosed by a Euclidean ball $\mB$ such that $\norm{\g}\le G$ for any
$\g\in\partial f(\u), \u\in\mB$. This matches the assumptions of the projection-based MOPES algorithm \cite{MOPES}.
Then, the additional localized bound $\overline G_0$ is not needed.
Indeed, every weak proximal output belongs to $\mK$, and hence so does
the average returned by each invocation of
Algorithm~\ref{alg:qg-sliding-stage}.  Thus every center $\c$ used by
Algorithm~\ref{alg:qg-nonsmooth-wpo} belongs to $\mK\subseteq\mB$.
For the auxiliary updates, replace the set
$\{\y|\norm{\y-\c}\le D\}$ in
Algorithm~\ref{alg:qg-sliding-stage} with
$\mY:=\{\y|\norm{\y-\c}\le D\}\cap\mB$.
The proof of Lemma~\ref{lem:qg-nonsmooth-sliding-stage} is unchanged.  Moreover,
all auxiliary first-order queries belong to $\mB$, while the points
$\v_\theta$ used in its proof belong to $\mK\subseteq\mB$.  
Therefore Algorithm~\ref{alg:qg-sliding-stage} may be run with
$\overline G_0=G$ and $K=1$.
Consequently, both the weak-proximal and subgradient complexities are
\[
 O\left(
  1+\frac{G^2}{\alpha\epsilon}
  +\log\left(
    1+\log\left(1+\frac{\alpha\Delta_0}{G^2}\right)
  \right)
 \right).
\]
\end{remark}

\section{Acknowledgement}
This work was funded by the European Union (ERC,  ProFreeOpt, 101170791). Views and opinions expressed are however those of the author(s) only and do not necessarily reflect those of the European Union or the European Research Council Executive Agency. Neither the European Union nor the granting authority can be held responsible for them.

\appendix

\section{Lower-Bound for Frank-Wolfe-type Methods}
The argument below is closely related to Lan's lower bound for nonsmooth
linear-optimization-based convex programming~\cite{lan2013complexity}, with the modification that it considers a strongly convex objective function.
\begin{theorem}%
\label{thm:lmo-nonsmooth-simplex-lb}
Fix $0<\alpha<G$ and $n\ge2$, and let $\simplex_n:=\left\{\x\in\reals_+^n:\ \sum_{i=1}^n x_i=1\right\}$.
Consider the following linear optimization model over $\simplex_n$: at every
linear-oracle call, on input $\c\in\reals^n$, the oracle returns some $\v\in\argmin_{\x\in\{\e_1,\dots,\e_n\}}\ip{\c}{\x}$,
with any fixed deterministic tie-breaking rule, and after $T$ such calls the
method must return a point
$\widehat\x\in\conv\{\v_1,\ldots,\v_T\}$. %
There exists a nonsmooth $\alpha$-strongly convex objective $f:\reals^n\to\reals$ with unique minimizer $\x^*\in\simplex_n$, such that every
subgradient of $f$ has norm at most $G$ over $\simplex_n$, and
for every $1\le T\le n$, every output of such a method satisfies
\begin{align}\label{eq:lmo-simplex-gap}
 f(\widehat\x)-f(\x^*)
 \ge (G-\alpha)\sqrt{\frac1T-\frac1n}.
\end{align}
Consequently, any such method guaranteeing error at most $\epsilon$ must use
at least $ T\ge \frac12\min\left\{n, (G-\alpha)^2/\epsilon^2\right\}$
linear optimization oracle calls.  %
\end{theorem}

\begin{proof}
Let
\[
 \u:=\frac1n\mathbf{1}
 \qquad\text{and}\qquad
 f(\x):=(G-\alpha)\norm{\x-\u}
       +\frac\alpha2\norm{\x-\u}^2.
\]
Note $f$ is indeed $\alpha$-strongly convex and  its unique minimizer is
$\x^*=\u\in\simplex_n$ with $f(\x^*)=0$.

We first verify the bounded-subgradient condition on $\simplex_n$.  For every
$\x\in\simplex_n$,
\[
 \norm{\x-\u}^2=\norm\x^2-\frac1n\le1-\frac1n<1.
\]
If $\x\ne\u$, then
\[
 \partial f(\x)
 =\left\{(G-\alpha)\frac{\x-\u}{\norm{\x-\u}}
                +\alpha(\x-\u)\right\},
\]
and hence every $\g\in\partial f(\x)$ satisfies
\[
 \norm\g
 =(G-\alpha)+\alpha\norm{\x-\u}
 \le G.
\]
At $\x=\u$ we have
$\partial f(\u)=(G-\alpha)\{\g:\norm\g\le1\}$, so the same bound holds.

Since every extreme point of $\simplex_n$ is a standard basis vector, after
$T$ linear-oracle calls every point in
$\conv\{\v_1,\ldots,\v_T\}$ has at most $T$ nonzero coordinates.  Let
$q:=|\supp(\widehat\x)|\le T$.  Since $\widehat\x\in\simplex_n$, the
Cauchy--Schwarz inequality gives
\[
 1=\left(\sum_{i\in\supp(\widehat\x)}\widehat x_i\right)^2
 \le q\sum_{i=1}^n\widehat x_i^2,
\]
and therefore $\norm{\widehat\x}^2\ge1/q\ge1/T$.  Using
$\ip{\widehat\x}{\u}=1/n=\norm\u^2$, we obtain
\[
 \norm{\widehat\x-\u}^2
 =\norm{\widehat\x}^2-\frac1n
 \ge\frac1T-\frac1n.
\]
Consequently,
\[
 f(\widehat\x)-f(\x^*)
 \ge (G-\alpha)\sqrt{\frac1T-\frac1n},
\]
which proves the theorem.

\end{proof}

\section{Proofs of Restricted Proximal Oracle Implementations}

\subsection{Proof of Lemma \ref{lem:symmetric-sparse-projection}}
We first restate the lemma and then prove it.

\begin{lemma}%
Let $\R$ satisfy the symmetry conditions above, and let $s\le r\le n$.
Fix $\y\in\reals^n$ and $\lambda>0$.

If $\dom(\R)\subseteq\reals_+^n$, let $T$ contain the indices of the $r$
largest entries of $\y$.  If $\R$ is invariant under coordinate-wise sign
changes, let $T$ contain the indices of the $r$ largest entries of $\y$ in
absolute value.  Let
\[
 \widehat{\u}
 \in
 \argmin_{\u\in\reals^r}
 \left\{
 \lambda\norm{\u-\y_T}^2+\R(\I_T\u)
 \right\},
 \qquad
 \widehat{\x}:=\I_T\widehat{\u}.
\]
Then
\[
 \widehat{\x}
 \in
 \argmin_{\x\in\mC_r}
 \left\{
 \lambda\norm{\x-\y}^2+\R(\x)
 \right\}.
\]
Consequently, if every $\x^*\in\mX^*$ is $s$-sparse, then for every
$r\ge s$ this construction implements a weak proximal oracle.
\end{lemma}

\begin{proof}
Fix any $\z\in\mC_r$. Consider first the case $\dom(\R)\subseteq\reals_+^n$.  By permuting the
coordinates of $\z$, we may place its nonzero entries on the coordinates
in $T$, pairing its largest entries with the largest entries of $\y$.
Denote the resulting point by $\widetilde{\z}$.  By permutation invariance,
$\R(\widetilde{\z})=\R(\z)$, $\norm{\widetilde{\z}}=\norm{\z}$.
Moreover, the rearrangement inequality gives
$\ip{\widetilde{\z}}{\y}
 \ge
 \ip{\z}{\y}$.
Hence,
\[
 \lambda\norm{\widetilde{\z}-\y}^2+\R(\widetilde{\z})
 \le
 \lambda\norm{\z-\y}^2+\R(\z).
\]

If $\R$ is invariant under coordinate-wise sign changes, first change the
signs of the nonzero entries of $\z$ to agree with the corresponding signs
of $\y$, and then permute their absolute values so that the largest ones are
paired with the $r$ largest entries of $\y$ in absolute value.  Signed
permutation invariance again preserves both $\R(\z)$ and $\norm{\z}$, while
the inner product with $\y$ cannot decrease.  The same inequality therefore
holds.

Thus, in either case, for every $\z\in\mC_r$ there is a point with support
contained in $T$ whose proximal objective is no larger.  It remains only to
minimize over points supported on $T$.  For $\x=\I_T\u$,
\[
 \norm{\x-\y}^2
 =
 \norm{\u-\y_T}^2+\norm{\y_{T^c}}^2,
\]
where the second term is independent of $\u$.  Hence the minimizer is
precisely $\widehat{\x}=\I_T\widehat{\u}$ defined above.

Finally, since $r\ge s$, every $\x^*\in\mX^*$ belongs to $\mC_r$.
Therefore, we indeed have that for every $\x^*\in\mX^*$, 
\[
 \lambda\norm{\widehat{\x}-\y}^2+\R(\widehat{\x})
 \le
 \lambda\norm{\x^*-\y}^2+\R(\x^*).
\]
\end{proof}

\subsection{Proof of Lemma \ref{lem:spectral-sparse-projection}}
We first restate the lemma and then prove it.

\begin{lemma}%
In the rectangular-matrix setting, suppose
$\R(\X)=\rho\bigl(\boldsymbol{\sigma}(\X)\bigr)$,
where $\rho:\reals_+^q\rightarrow\reals\cup\{+\infty\}$ is permutation
invariant.  Fix $\Y\in\reals^{d_1\times d_2}$ and $\lambda>0$, and let
$\Y
 =
 \U\operatorname{Diag}(\boldsymbol{\sigma})\V^\top$,
 $\sigma_1\ge\cdots\ge\sigma_q\ge0$, be a SVD, where $\U$ and $\V$ have $q$ orthonormal
columns. Let $\U_r$ and $\V_r$ contain the
first $r$ columns of $\U$ and $\V$, respectively.  Let
\[
 \widehat{\u}
 \in
 \argmin_{\u\in\reals_+^r}
 \left\{
 \lambda\norm{\u-\boldsymbol{\sigma}_{1:r}}^2
 +
 \rho\!\left((\u,\mathbf{0}_{q-r})\right)
 \right\},
\]
and set
$\widehat{\X}
 :=
 \U_r\operatorname{Diag}(\widehat{\u})\V_r^\top$.

In the symmetric-matrix setting, suppose
$\R(\X)=\rho\bigl(\boldsymbol{\lambda}(\X)\bigr)$,
where $\rho:\reals^d\rightarrow\reals\cup\{+\infty\}$ is permutation
invariant and $\dom(\rho)\subseteq\reals_+^d$.  Fix
$\Y\in\mbS^d$ and $\lambda>0$, and let
$ \Y
 =
 \U\operatorname{Diag}(\boldsymbol{\gamma})\U^\top$,
 $\gamma_1\ge\cdots\ge\gamma_d$,
be an eigendecomposition.  Let $\U_r$ contain the first $r$ columns of
$\U$.  Let
\[
 \widehat{\u}
 \in
 \argmin_{\u\in\reals_+^r}
 \left\{
 \lambda\norm{\u-\boldsymbol{\gamma}_{1:r}}^2
 +
 \rho\!\left((\u,\mathbf{0}_{d-r})\right)
 \right\},
\]
and set
$\widehat{\X}
 :=
 \U_r\operatorname{Diag}(\widehat{\u})\U_r^\top$.

In either setting,
\[
 \widehat{\X}
 \in
 \argmin_{\X\in\mC_r}
 \left\{
 \lambda\norm{\X-\Y}_F^2+\R(\X)
 \right\}.
\]
Consequently, if every $\X^*\in\mX^*$ has rank at most $r$, the
corresponding construction implements a weak proximal oracle.
\end{lemma}

\begin{proof}
We first consider the rectangular-matrix setting.  Fix any
$\X\in\mC_r$, and let
$\tau_1\ge\cdots\ge\tau_r\ge0$ denote its nonzero singular values, padded
with zeros if $\rank(\X)<r$.  Define
$\widetilde{\X}
 :=
 \U_r\operatorname{Diag}(\boldsymbol{\tau}_{1:r})\V_r^\top$.
The matrices $\X$ and $\widetilde{\X}$ have the same singular values, and
therefore
$\R(\widetilde{\X})=\R(\X)$.
Moreover, von Neumann's trace inequality gives
$\langle{\X,\Y\rangle}
 \le
 \sum_{i=1}^r\tau_i\sigma_i
 =
 \langle{\widetilde{\X},\Y}\rangle$.
Since $\X$ and $\widetilde{\X}$ also have the same Frobenius norm,
$\Vert{\widetilde{\X}-\Y}\Vert_F^2
 \le
 \Vert{\X-\Y}\Vert_F^2$.
Thus, every rank-$r$ candidate can be replaced, without increasing the
proximal objective, by one whose singular vectors agree with the leading
singular vectors of $\Y$ and whose diagonal entries are nonincreasing.

It remains to observe that the ordering constraint on these diagonal entries
is redundant.  For any $\u\in\reals_+^r$, let $\u^\downarrow$ denote the
vector obtained by sorting its entries in nonincreasing order.  Permutation
invariance of $\rho$ gives
$\rho\!\left((\u^\downarrow,\mathbf{0}_{q-r})\right)
 =
 \rho\!\left((\u,\mathbf{0}_{q-r})\right)$.
Since $\boldsymbol{\sigma}_{1:r}$ is nonincreasing, the rearrangement
inequality gives
$\langle{\u^\downarrow, \boldsymbol{\sigma}_{1:r}}\rangle
 \ge
 \langle{\u, \boldsymbol{\sigma}_{1:r}}\rangle$,
and hence
$\norm{\u^\downarrow-\boldsymbol{\sigma}_{1:r}}^2
 \le
 \norm{\u-\boldsymbol{\sigma}_{1:r}}^2$.
Therefore, minimizing over all $\u\in\reals_+^r$ gives the same optimal value
as minimizing under the additional constraint
$u_1\ge\cdots\ge u_r\ge0$.

Finally, for every $\u\in\reals_+^r$, permutation invariance of $\rho$ and
orthogonality of the singular vectors give
\[
 \R\!\left(\U_r\operatorname{Diag}(\u)\V_r^\top\right)
 =
 \rho\!\left((\u,\mathbf{0}_{q-r})\right)
\]
and
\[
 \Vert{\U_r\operatorname{Diag}(\u)\V_r^\top-\Y}\Vert_F^2
 =
 \norm{\u-\boldsymbol{\sigma}_{1:r}}^2
 +
 \sum_{i=r+1}^q\sigma_i^2.
\]
The last sum is independent of $\u$.  Hence the scalar problem defining
$\widehat{\u}$ is exactly the remaining part of the rank-$r$ proximal
problem, and $\widehat{\X}$ is a minimizer.

We next consider the symmetric-matrix setting.  Fix any $\X\in\mC_r$.
Since $\dom(\R)\subseteq\mbS_+^d$, the matrix $\X$ is positive
semidefinite.  Let
$\tau_1\ge\cdots\ge\tau_r\ge0$ denote its nonzero eigenvalues, padded with
zeros if $\rank(\X)<r$, and define
$ \widetilde{\X}
 :=
 \U_r\operatorname{Diag}(\boldsymbol{\tau}_{1:r})\U_r^\top$.
The matrices $\X$ and $\widetilde{\X}$ have the same eigenvalues, and
therefore
$\R(\widetilde{\X})=\R(\X)$.
The trace inequality for symmetric matrices gives
$\langle{\X,\Y}\rangle
 \le
 \sum_{i=1}^r\tau_i\gamma_i
 =
 \langle{\widetilde{\X},\Y}\rangle$.
Since $\X$ and $\widetilde{\X}$ also have the same Frobenius norm,
$\Vert{\widetilde{\X}-\Y}\Vert_F^2
 \le
 \Vert{\X-\Y}\Vert_F^2$.
Thus, every rank-$r$ candidate can be replaced, without increasing the
proximal objective, by one whose eigenvectors agree with the leading
eigenvectors of $\Y$ and whose nonzero eigenvalues are nonincreasing.

As above, the ordering constraint is redundant.  For any
$\u\in\reals_+^r$, permutation invariance of $\rho$ and the rearrangement
inequality give
$\rho\!\left((\u^\downarrow,\mathbf{0}_{d-r})\right)
 =
 \rho\!\left((\u,\mathbf{0}_{d-r})\right)$
and
$ \norm{\u^\downarrow-\boldsymbol{\gamma}_{1:r}}^2
 \le
 \norm{\u-\boldsymbol{\gamma}_{1:r}}^2$.
Moreover,
\[
 \R\!\left(\U_r\operatorname{Diag}(\u)\U_r^\top\right)
 =
 \rho\!\left((\u,\mathbf{0}_{d-r})\right)
\]
and
\[
 \Vert{\U_r\operatorname{Diag}(\u)\U_r^\top-\Y}\Vert_F^2
 =
 \norm{\u-\boldsymbol{\gamma}_{1:r}}^2
 +
 \sum_{i=r+1}^d\gamma_i^2.
\]
The last sum is independent of $\u$.  Hence, the scalar problem defining
$\widehat{\u}$ is exactly the remaining part of the rank-$r$ proximal
problem, and $\widehat{\X}$ is again a minimizer.

Finally, if every $\X^*\in\mX^*$ has rank at most $r$, then
$\X^*\in\mC_r$, and therefore, in either setting,
\[
 \lambda\Vert{\widehat{\X}-\Y}\Vert_F^2+\R(\widehat{\X})
 \le
 \lambda\Vert{\X^*-\Y}\Vert_F^2+\R(\X^*).
\]
\end{proof}

\bibliography{bibs}
\bibliographystyle{plain}

\end{document}